\documentclass[11pt]{article}
\usepackage[colorlinks]{hyperref}
\usepackage{color}
\usepackage{graphicx}
\usepackage{graphics}
\usepackage{makeidx}
\usepackage{showidx}
\usepackage{latexsym}
\usepackage{amssymb}
\usepackage{verbatim}
\usepackage{amsmath}
\usepackage{amsthm}
\usepackage{amsfonts}
\usepackage{subcaption} 
\usepackage[abs]{overpic}
\usepackage{xcolor,varwidth}
\usepackage[title]{appendix}
\usepackage{leftidx}
\usepackage{blindtext}
\usepackage{tabularx}
\usepackage{titlesec}
\usepackage[mathscr]{euscript}
\usepackage{pgfplots}
\usepackage{array, ltablex, multirow}
\usepackage{tikz-cd} 

\usepackage{geometry}
\newtheorem{theorem}{Theorem}[section]
\newtheorem{proposition}[theorem]{Proposition}
\newtheorem{definitionp}[theorem]{Definition/Proposition}
\newtheorem{remark}[theorem]{Remark}

\newtheorem{definition}[theorem]{Definition}
\newtheorem{corollary}[theorem]{Corollary}
\newtheorem{lemma}[theorem]{Lemma}

\newtheorem{convention}[theorem]{Convention}
\usepackage{enumerate,amssymb, xcolor,mathrsfs}

\newtheorem{mainthm}{Theorem}

\title{Strongly Adapted Contact Geometry of Anosov 3-Flows}

\author{Surena Hozoori}
\newcommand{\Addresses}{{
  \bigskip
  \footnotesize

Surena Hozoori, \textsc{Department of Mathematics, University of Rochester.}\par\nopagebreak
  \textit{E-mail address}: \texttt{shozoori@ur.rochester.edu}
  
  }}
    \date{}
\begin{document}
\maketitle
\noindent
\begin{abstract}
We provide a 3 dimensional contact geometric characterization of Anosov 3-flows based on interactions with Reeb dynamics. We investigate basic properties of the space of the resulting geometries and in particular show that such space is homotopy equivalent to the space of Anosov 3-flows. A technical theorem on the asymptotic synchronization of adapted norms is proved, which can be of broader interest.
\end{abstract}

{
  \hypersetup{linkcolor=black}
  \tableofcontents
}


\section{Introduction}

Hyperbolic dynamics and low dimensional topology have been two of the most rapidly growing fields of mathematics since the mid 20th century, branching out into many sub theories, revealing deep interplays with other areas of research. In particular, the study of {\em Anosov flows in dimension 3} has provided a framework where the two theories interact most profoundly and colorfully. 
With one foot in the theory of chaos and another in geometric topology, {\em Anosov dynamics} is proven to be a subtle multi-faceted theory in the intersection of dynamics, geometry and topology. More specifically, most of these advances have occurred since the 1970s and thanks to the very successful developments and applications of {\em taut foliations} and {\em geometric group theory} \cite{fenley,danny,bartintro,hflows}.

On the other hand, rooted in the Hamiltonian reformulation of Newtonian physics, the study of {\em contact structures} has appeared in various geometric contexts throughout the previous century and since the 1980s, far-reaching relations between these geometric structures and topology have been explored, especially in low dimensions. As a result, {\em contact topology} has bloomed into one of the most active areas of low dimensional topology in the past decades \cite{geiges,martinet}. Moreover, a dynamical approach to contact topology has been well established via the study of {\em Reeb flows} associated with contact structures \cite{hofer}.

Since their introductions, interplays between Anosov dynamics and contact topology have been observed commonly, but mostly in the context of {\em geodesic flows} and their generalizations. The foundation for a more general and deeper investigation of their connections however, goes back to a novel interplay between the two theories, discovered in the mid 1990s, by Mitsumatsu \cite{mit1} and Eliashberg-Thurston \cite{et}, motivated by their study of {\em exotic Liouville geometry in dimension 4} and {\em perturbation theory of foliations}, respectively. Most importantly, they showed that an important generalization of Anosov dynamics can be characterized, in dimension 3, in purely contact geometric terms. Such class of dynamics with a weak notion of hyperbolicity is called {\em projectively Anosov flows}, or {\em flows with dominated splittings}, and the contact geometric counterpart is named {\em bi-contact structure}. Mitsumatsu \cite{mit1} takes this another step forward by constructing 4-dimensional Liouville domains with disconnected boundary, based on Anosov dynamics in dimension 3 (see Section~\ref{s2} for relevant definitions and history). 

These ideas were left mostly unexplored until recent years, where significant progress has occurred in the landscape of theory, mainly motivated by the {\em sympletic geometric theory of Anosov flows} \cite{bbp,hoz3,sal1,clmm,hoz6} and {\em convex hypersurface theory} \cite{huang,honda,bc}. In particular in \cite{hoz3}, the author exploits the construction of Mitsumatsu to introduce a purely contact and symplectic geometric characterization of Anosov 3-flows. This means that Anosov 3-flows can in fact be viewed as contact and symplectic geometric objects and hence, all the tools in the study of contact and symplectic geometry can be carried over to explore questions about Anosov flows. This results in a contact geometric theory of Anosov 3-flows with a strong 4-dimensional symplectic flavor. Subsequently, Cieliebak-Lazarev-Massoni-Moreno \cite{clmm} introduced a family of novel symplectic geometric invariants for these flows, based on the subtle geometric analysis of {\em Floer theory}.

While the depth of the theory, as described above, lies mainly in 4 dimensional symplectic topology, further interactions with 3 dimensional contact topology have been proven to exist in parallel, mainly relying on interplays with Reeb dynamics. This includes the development of a {\em contact geometric surgery theory of Anosov 3-flows} by Foulon-Hasselblatt \cite{fh,hass} and Salmoiraghi \cite{sal1,sal2}, presenting a unifying contact geometric approach to surgeries of Anosov flows, a very important ingredient in the classical theory. Furthermore, the existence of invariant volume forms and the dynamical consequences are studied from the viewpoint of bi-contact geometry and Reeb dynamics in \cite{hoz4}.

\vskip0.5cm

\textbf{1.1 The Main result:}
 The main goal of this paper is to provide a purely contact geometric characterization of Anosov 3-flows, based on their interactions with Reeb dynamics. This results in a contact geometric local picture for these flows, based on which we set up a novel contact geometric framework for Anosov dynamics, recontextualizing the previous related work.
The following is the main theorem of this paper. The reader should consult Section~\ref{s2} for definitions related to bi-contact geometry and its relation to Anosov dynamics.

 \begin{mainthm}\label{A}
 Let $X$ be a $C^\infty$ non-vanishing vector field on a closed oriented 3-manifold $M$. TFAE:
 
 (1) $X$ generates a transversally oriented Anosov flow;
 
 (2) there exists a positive contact form $\alpha_+$, such that $\alpha_+(X)=0$ and $\mathcal{L}_X\alpha_+$ is a negative contact form;
 
 (3) there exists a bi-contact structure $(\xi_-,\xi_+)$ supporting $X$, such that $\xi_-$ contains a Reeb vector field of $\xi_+$;
 
 (4) there exists a Reeb flow $R_+$ for a positive contact structure containing $X$, such that 
 
 \noindent $(X,[R_+,X],R_+)$ is an oriented basis for $TM$.
 
 \end{mainthm}
 
\vskip0.05cm

In the rest of this paper, unless stated otherwise, we assume that $M$ is a closed oriented connected 3-manifold, $X$ is a smooth (i.e. $C^\infty$) vector field on $M$ and $X^t$ is its generated flow. Also, we assume all (projectively) Anosov flows to be {\em transversally oriented}, i.e. have oriented invariant bundles. This condition is always satisfied, possibly after lifting to a double cover of $M$. Even though $C^2$-regularity is sufficient for our main results, we take all the vector fields and contact structures in this paper to be $C^\infty$, unless stated otherwise, for the sake of a smoother presentation. One can further refine the results in terms of regularity.

\vskip0.5cm
 
The conditions (2), (3) and (4) in the above theorem provide different flavors of our characterization in terms of contact form geometry, bi-contact geometry and Reeb dynamics, respectively.  We call a contact form, bi-contact structure or Reeb flow {\em strongly adapted} to $X$, if it satisfies the condition (2), (3) or (4), respectively, in Theorem~\ref{A}.

Condition (2) implies that the Anosovity of $X$ can be described as a quite elementary open local condition on a {\em Legendrian}, i.e. tangent, flow of a contact structure. This in particular results in a contact geometric local model for Anosov 3-flows. While the foliation theoretic local picture of these flows has been proven to be extremely effective in answering question in Anosov dynamics, we hope that adopting a contact geometric description could enhance the theory by application of new elements from contact topology, a theory rich in analytical and topological aspects. In particular, subtle geometric analytical theories (e.g. Floer theories) can be applied in this setting, since the contact geometric model enjoys regularity as high as the generating vector field, thanks to the openness of the contact condition. This is contrary to the foliations underlying Anosov flows which are often of low regularity, and as a result, the methods applied in the classical foliation theoretic approach are very topological in nature. Furthermore, perturbations of our geometric model in a neighborhood does not affect the geometry of the flow far away, a sharp contrast to the invariant foliations which are determined by the long term behavior of the flow. This has often created technical difficulty in important aspects of Anosov dynamics related to surgery and gluing operations, as it is natural to require a suitable deformation theory. As contact structures are {\em stable} structures themselves, unlike integrable plane fields tangent to the invariant foliations, we hope to develop a deformation theory based on contact geometry, which is compatible with the Anosov setting.

Condition (3) puts our characterization in the context of the bi-contact geometric approach initiated by Mitsumatsu and Eliashberg-Thurston \cite{mit1,et} (see Section~\ref{s2} for more in-depth account). As mentioned above, they characterized projective Anosovity of a 3-flow by the existence of a {\em supporting} bi-contact structure $(\xi_-,\xi_+)$. Theorem~\ref{A} improves the dictionary between notions in Anosov dynamics and bi-contact geometric conditions, by describing the additional contact geometric condition required to achieve a correspondence for Anosov flows. That is, $\xi_-$ containing a Reeb vector field of $\xi_+$. This is in fact in the same spirit of the characterization of incompressible Anosov 3-flows we previously introduced in \cite{hoz4}, where the necessary and sufficient condition is both $\xi_-$ and $\xi_+$ containing Reeb vector fields of one another, simultaneously, a condition we call being {\em strongly bi-adapted}. This characterization is established in $C^1$ category in \cite{hoz4}. But we argue in Remark~\ref{biadaptedrem} that it actually holds in the $C^\infty$ category. This can essentially be viewed as Theorem~\ref{A}, in the presence of an invariant volume form. The incompressibility condition implies symmetries in the geometry of a supporting bi-contact structure, allowing us to bypass the main technical difficulties in the proof of Theorem~\ref{A}, i.e. the {\em asymptotic synchronization process}. 
 
We summarize the state of affairs in the following table. The row for projectively Anosov flows describes the observation of Mitsumatsu and Eliashberg-Thurston and Theorem~\ref{A} is responsible for the row for Anosov flows. The {\em domination relation} in the following table, denoted by $\prec$, is introduced in Section~\ref{s2} and we refer the reader to \cite{gour} for a thorough discussion.

\vskip0.1cm
 
\begin{tabularx}{1\textwidth} { 
  | >{\centering\arraybackslash}X
    | >{\centering\arraybackslash}X 
  | >{\centering\arraybackslash}X 
  | >{\centering\arraybackslash}X | }
  \caption{Dictionary between notions in Anosov dynamics and bi-contact geometric conditions}\label{tab:1}\\
 \hline
\textbf{dynamical condition} & \textbf{Anosov dynamics} & \textbf{bi-contact geometry} & \textbf{geometric condition} \\
 \hline
 \hline
$TM/\langle X\rangle \simeq E\oplus F$ $X_*|_E\prec X_*|_F$  & Projectively Anosov flows  & Bi-contact structures & $X\subset \xi_-\pitchfork \xi_+$  \\
\hline
$TM\simeq \langle X\rangle \oplus E^s\oplus E^u$ $X_*|_{E^s} \prec 0 \prec X_*|_{E^u}$  & Anosov flows  & Strongly adapted bi-contact structures & $X\subset  \ker{\alpha_-} \pitchfork \ker{\alpha_+}$ $\alpha_-(R_{\alpha_+})=0$\ \\
\hline
$TM\simeq \langle X\rangle \oplus E^s\oplus E^u$ $||X_*|_{E^s}|| =||- X_*|_{E^u}||$  & Incompressible Anosov flows  & Strongly bi-adapted bi-contact structures & $X\subset  \ker{\alpha_-} \pitchfork \ker{\alpha_+}$ $\alpha_-(R_{\alpha_+})=\alpha_+(R_{\alpha_-})=0$ \\
\hline

\end{tabularx}

\vskip1cm

Finally, condition (4) provides a Reeb dynamical approach to the study of Anosov 3-flows. We note that the Reeb vector field $R_+$ in Theorem~\ref{A} determines the underlying bi-contact structure as $(\xi_-,\xi_+)=(\langle X,R_+ \rangle,\langle X,[X,R_+] \rangle)$. In other words, this defines Anosov dynamics as a certain interplay between a pair of Legendrian and Reeb flows $(X,R_+)$ of a contact structure.

\begin{quote}
\centering
{\em Anosov 3-flows are Legendrian flows with a special property: they form a contact structure of opposite orientation, when paired with some Reeb vector field of the underlying contact structure.}
\end{quote}

Almost all the literature on Anosov dynamics is from the view point of the Legendrian vector field $X$. This paper is an invitation to study this interplay from the viewpoint of the Reeb dynamics of $R_+$. In order to understand the local description of Theorem~\ref{A} better, note that given a flow segment $\gamma$ of an Anosov flow $X^t$ and a strongly adapted Reeb vector field $R_+$, one observes distinct behaviors of nearby flow lines when moving in the directions of $R_+$ or $[X,R_+]$ (see Figure~1). In the classical description, these two strands of flow lines are separated by the invariant foliations determined by the long term behavior of the underlying Anosov flow. In our theory however, the two are separated by a geometric condition, i.e. the {\em strong adaptation} of $X$ and $R_+$.





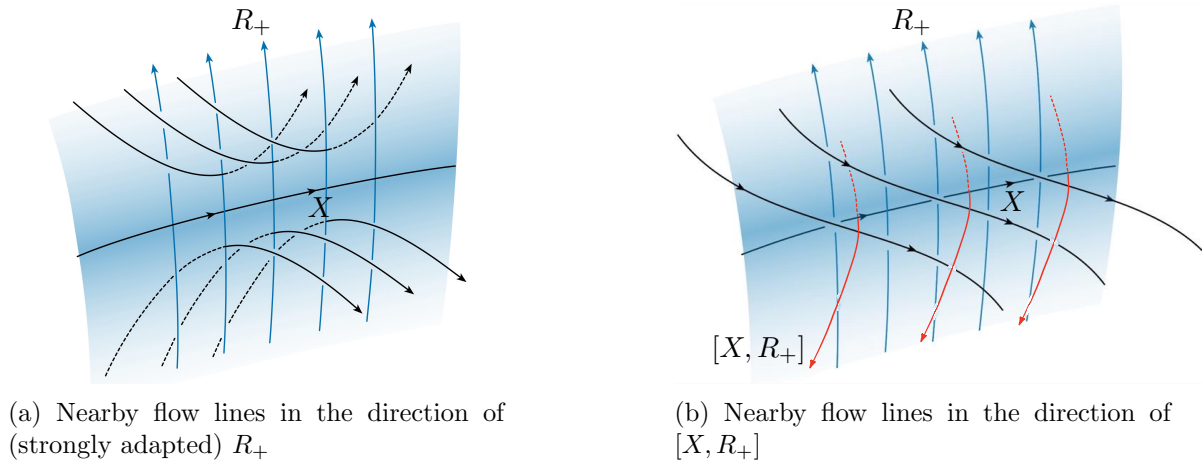
\begin{figure}[h]\label{localanosov}

 \begin{subfigure}[b]{0.4\textwidth}

  \center \begin{overpic}[width=5.6cm]{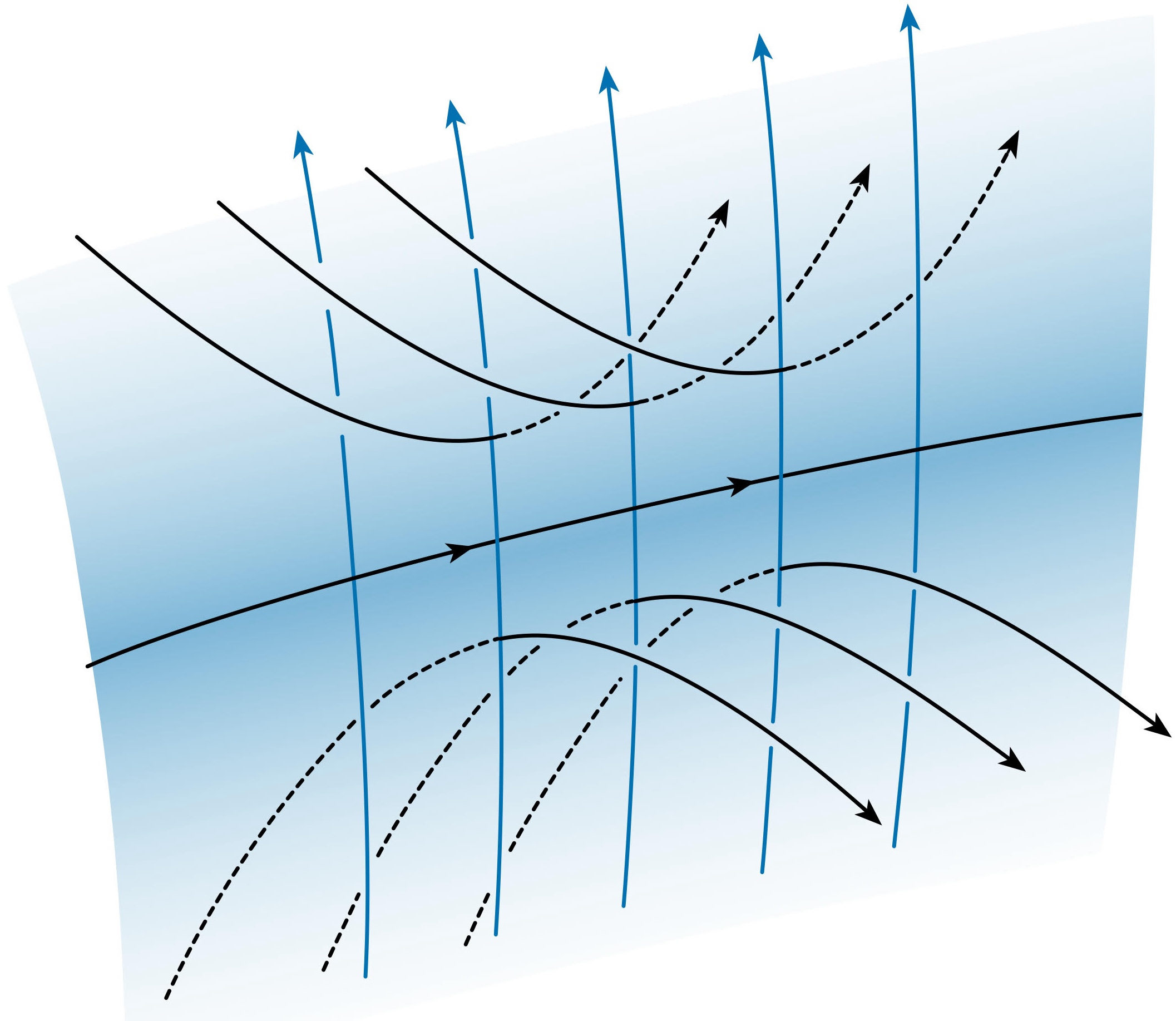}

  \put(70,134){$R_+$}
    \put(99,63){$X$}
       \end{overpic}
    \caption{Nearby flow lines in the direction of (strongly adapted) $R_+$}
    \label{fig:1}
  \end{subfigure}
  \hspace{2cm}
  \begin{subfigure}[b]{0.4\textwidth}
  \center \begin{overpic}[width=7cm]{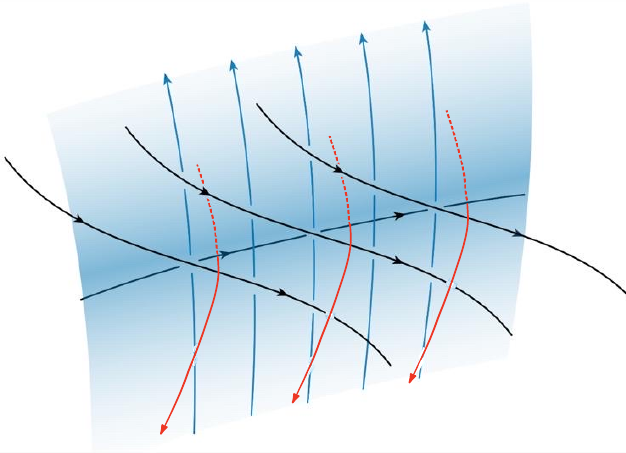}
  \put(82,134){$R_+$}
    \put(123,65){$X$}
    \put(14,10){$[X,R_+]$}

  \end{overpic}
    \caption{Nearby flow lines in the direction of $[X,R_+]$}
    \label{fig:2}
  \end{subfigure}
  
 \caption{Local behavior of flow lines in in the strongly adapted contact geometry of Anosov flows}
\end{figure}

The local picture above appears naturally in the case of algebraic geodesic flows, and in the neighborhood of periodic orbits lying on {\em quasi-transverse tori}, which are foliated by the standard cyclic Reeb flows in those examples. In fact, this local description, together with the surgery techniques of Foulon-Hasselblatt \cite{fh}, is already used extensively by Salmoiraghi \cite{sal1,sal2}, in order to develop a bi-contact geometric theory of (projectively) Anosov surgeries. The flexible and high regularity setting of strongly adapted contact geometry allows Salmoiraghi to introduced new contact geometric surgeries, unifying various classical surgery operations, which were designed to resolve historical questions in Anosov dynamics. One can see that Salmoiraghi's surgery operation in \cite{sal2} preserves the strongly adapted geometry and hence, paired with Theorem~\ref{A}, it {\em localizes} the effect of gluing operations and the required deformations on the dynamics of the flow surgeries.

It is noteworthy that as a corollary, Theorem~\ref{A} allows one to make the symplectic characterization at the core of the 4-dimensional theory \cite{hoz3,massoni,clmm,hoz6} more efficient in the following sense. In the symplectic theory, Anosovity is established by two independent {\em Liouville conditions} on a pair of positive and negative contact forms $(\alpha_-,\alpha_+)$. By Theorem~\ref{A}, one can always set $\alpha_-:=\mathcal{L}_X\alpha_+$ for an appropriate choice of $\alpha_+$ and simplify the necessary Liouville condition to only one.

\begin{corollary}
The projectively Anosov flow $X^t$ is Anosov, if and only if, there exists a Liouville domain (manifold) of the type
$$\alpha:=(1-s)\mathcal{L}_X \alpha_+ +(1+s)\alpha_+ \ \ \ \ (:=e^{-s}\mathcal{L}_X \alpha_+ +e^s\alpha_+) \ \ \ \ \ \ 
\text{on}\ \  [-1,1]_s\times M\ \ \ \  (\mathbb{R}_s\times M),$$ for some positive contact form $\alpha_+$ such that $\alpha_+(X)=0$.
\end{corollary}

\vskip0.5cm

\textbf{1.2. The space of strong adaptations:} For many purposes, which includes deriving new contact geometric invariants from our local model, we need to study the space of the geometries described in Theorem~\ref{A}. We denote the space of Anosov flows on $M$, up to (unoriented) reparametrization, by ${\mathcal{A}}(M)$ and define the space  of {\em strong adaptations} on $M$, denoted by $\mathcal{SA}^+(M)$. See Section~\ref{s5} for a thorough discussion and in particular, Corollary~\ref{saspace} for the following equivalence (we naturally identify contact forms up to constant scaling).
$$\mathcal{SA}^+(M):=\bigg\{(\alpha_+,\langle X \rangle)| \alpha_+ \text{ is a strongly adapted contact form for }\langle X \rangle \in \mathcal{A}(M) \bigg\}/ (\mathbb{R}\backslash\{0\},.)$$
$$ \simeq \bigg\{(R_+,\langle X \rangle)| R_+ \text{ is a strongly adapted Reeb flow for }\langle X \rangle\in \mathcal{A}(M) \bigg\}/ (\mathbb{R}\backslash\{0\},.)$$
$$ \simeq \bigg\{(\xi_-,\xi_+)| (\xi_-,\xi_+) \text{ is a strongly adapted bi-contact structure for } \langle \xi_-\pitchfork\xi_+ \rangle\in \mathcal{A}(M) \bigg\} .$$

We investigate various fibrations of the space of strong adaptations and summarize our study in the following theorem. Among other things, this implies that the space of strong adaptations is homotopy equivalent to the space of Anosov flows. This is in particular important for various classification problems, since it allows one to exploit the contact geometric invariants of strong adaptations as homotopy invariants of Anosov 3-flows and study the homotopy type of the space of Anosov 3-flows. For the 4-dimensional symplectic geometric theory of Anosov flows, a similar theorem was proved by Massoni \cite{massoni} in 2022, which was later refined and generalized in \cite{hoz6}. Here, by {\em acyclic} fibration, we mean one with contractible fibers.

\begin{mainthm}\label{B}
The map $$\begin{cases}
\pi_A:\mathcal{SA}^+(M)\rightarrow {\mathcal{A}}(M)\\
(R_{+},\langle X \rangle)\mapsto \langle X \rangle
\end{cases}$$ is an acyclic fibration. In particular, $\mathcal{SA}^+(M)$ is homotopy equivalent to ${\mathcal{A}}(M)$.
\end{mainthm}

It is noteworthy that the above fibration structure can be refined by understanding its interaction with other spaces related to the contact geometry of strong adaptations and we leave the more involved statement to Corollary~\ref{refine}.

\begin{remark}
Notice that strong adaptations can be easily defined on manifolds with boundary or non-compact manifolds and subsequently, they define a notion of hyperbolicity in the dynamics of those manifolds. There are other classical notions of hyperbolic dynamics in such categories ({\em uniform hyperbolicity}, {\em hyperbolic plugs}, etc.) and the relation between those and strong adaptations remains unknown at this point. For instance on $\mathbb{R}^3$, one can easily check that the pair $( R_+:=\partial_z,\langle X\rangle:=\langle z\partial_x -\partial_y+yz\partial_z \rangle )$ satisfies the local condition of strong adaptations.
\end{remark}

\vskip0.5cm
 
\textbf{1.3. Asymptotic synchronization process:} Finally, we would like to highlight that a technical ingredient in the proof of Theorem~\ref{A} is an {\em asymptotic synchronization} process we apply to the invariant bundles of an Anosov flow. More specifically, we show that for a general Anosov flow, one can assume the expansion and contraction in the unstable and stable invariant bundles happen at {\em almost} constant rates. This process can be interpreted as {\em almost synchronization} of an Anosov flow via deformation of the norm $||.||$ on $TM$. This is while synchronization, i.e. having constant expansion rates, is only possible in the incompressible case and possibly after a reparametrization of the flow. The main idea is to use an averaging technique introduced by Holmes \cite{holmes} and popularized by Hirsch-Pugh-Shub \cite{invariant}. They use this technique to show that any Anosov system admits an {\em adapted metric}, which can be thought of as a weaker notion of {\em synchronizing} a norm along the flow. Their result was then generalized to various settings by Gourmelon \cite{gour}. We prove an asymptotic synchronization theorem for flow actions on vector bundles, which gives a strong refinement of the result of Hirsch-Pugh-Shub \cite{invariant} and Gourmelon \cite{gour} for adapted norms in this setting. We show that the norm in the definition of an Anosov flow can be deformed to have expansion rates arbitrarily close to the {\em Lyapunov exponents} at each point, with derivatives along the flow vanishing exponentially fast during the deformation. We believe this theorem to be of broader use. This theorem is in fact already exploited in \cite{hoz5} to compute the infimum of a the {\em Dirichlet energy} of a contact Anosov flow in terms of its {\em Liouville entropy}, contributing to a classical question in the theory of {\em contact metrics}. We defer the statement of this technical theorem to Theorem~\ref{expunif} in Section~\ref{s4}, where the necessarily notions and discussions are brought beforehand.


Applying Theorem~\ref{expunif} to the invariant bundles of an Anosov 3-flow allows one to deform the supporting contact geometry in order to make it asymptotically close to the contact geometry of an Anosov 3-flow with constant divergence. Then, the contact geometric characterization of Theorem~\ref{A} is obtained by applying ideas similar to ones used in the characterization of the incompressible case in \cite{hoz4}.

Theorem~\ref{expunif} in fact yields a much stronger convergence than the one needed in the proof of Theorem~\ref{A}. This results in a refinement of the space of strong adaptation by assuming {\em extra degrees of adaptations}. This means that the contact geometric model of an arbitrary Anosov flow can be set up to be arbitrarily close, in terms of its variations along the flow, to one with synchronized invariant bundles. 

\begin{corollary}\label{refcor}
We have the filteration of $\mathcal{SA}^+(M)$ by non-empty subspaces as below.
 $$...\ \mbox{\Large$\subset$}\bigg\{\substack{\text{Strongly adapted $\alpha_+$} \\  \text{Strongly adapted $\mathcal{L}_X\alpha_+$} \\ \text{Strongly adapted $\mathcal{L}_X\mathcal{L}_X \alpha_+$}} \bigg\} \mbox{\Large$\subset$}
 \bigg\{\substack{\text{Strongly adapted $\alpha_+$} \\  \text{Strongly adapted $\mathcal{L}_X\alpha_+$}} \bigg\} \mbox{\Large$\subset$}
\  \mathcal{SA}^+(M).$$
\end{corollary}

The non-emptiness claim in Corollary~\ref{refcor} in particular means that depending on the application, the contact geometric characterization of Anosov flows, as in Theorem~\ref{A}, can be based on {\em any layer} of the above hierarchy. In fact, even further refinements of the above is possible by introducing a weaker notion of {\em adaptation}. For instance in this paper, we extend our study to the case when $\alpha_+$ and $\mathcal{L}_X\alpha_+$ are strongly adapted and {\em adapted}, respectively. Subsequently, we show that the fibration of such geometries is fiberwise convex, which can be useful in computations. We avoid overpopulation of this introduction and refer the reader to the final remarks of Section~\ref{s3} and Section~\ref{s5} (in particular, Theorem~\ref{mains}, Remark~\ref{strongfilter} and Corollary~\ref{sfiber}) for further discussion.

\vskip0.5cm

\textbf{Organization:} In Section~\ref{s2}, we provide the necessary background in the contact geometric theory of Anosov 3-flows. Section~\ref{s3} is devoted to investigating the Reeb dynamics of bi-contact structures, which results in the introduction of our contact geometric model and the proof of our main theorem (Theorem~\ref{A}), modulo a technical difficulty left to be dealt with in the next section. That is, contextualizing and proving the asymptotic synchronization theorem (Theorem~\ref{expunif}) in Section~\ref{s4}. In Section~\ref{s5}, we study the space of strong adaptations and we summarize our observations in a fibration theorem (Theorem~\ref{B}).

\vskip0.5cm
\textbf{ACKNOWLEDGEMENT:} We are grateful to Federico Salmoiraghi, conversations with whom in 2023 initiated many of the ideas presented here, for many fruitful exchanges. We also thank Salmoiraghi for the very generous offer of help with most of the figures in this paper. Moreover, we greatly appreciate Peter Albers, Jonathan Bowden and Agustin Moreno, the organizers of the {\em Symplectic geometry and Anosov dynamics workshop}, held in July 2024 in Heidelberg, Germany, for providing a great opportunity to communicate the ideas in this paper, and an excuse to publish them. We further thank Moreno for commenting on an earlier version of this paper.

\section{Bi-contact geometry and projective Anosovity}\label{s2}

Recall that on a (closed oriented) 3 dimensional manifold $M$, an {\em Anosov flow} $X^t$ is a flow admitting a continuous $X^t$-invariant splitting $TM\simeq \langle X \rangle \oplus E^s \oplus E^u$ such that for any $t\in\mathbb{R}$
$$\begin{cases}
||X^t_*e_s||\leq A e^{-Ct}||e_s||\text{ for any } e_s\in E^s \\
||X^t_*e_u||\geq A e^{Ct}||e_u||\text{ for any } e_u\in E^u
\end{cases},$$
for some constants $A,C>0$ and some norm $||.||$ on $TM$. This means that the norms on the line bundles $E^u$ and $E^s$, namely the {\em strong unstable and stable bundles}, have {\em eventual exponential expansion and contraction}, respectively. However, using an averaging technique of Holmes \cite{holmes} and Hirsch-Pugh-Shub \cite{invariant}, one can find another norm on $TM$ with respect to which {\em immediate exponential expansion and contraction} can be observed. That is equivalent to having $A=1$ in the above {\em Anosov condition} with respect to such norm. Norms with this property are called {\em adapted} and constitute a convex (see Lemma~\ref{adnormconvex}) set of norms on the invariant bundles.

Classical examples of Anosov 3-flows are the suspension of an (orientation preserving) {\em Anosov automorphism} of $\mathbb{T}^2$ with a constant roof function, or the geodesic flow induced on the unit tangent bundle $UT\Sigma$ associated with a hyperbolic surface or orbifold $\Sigma$. Anosov flows finitely covered, in a differentiable way, by these examples are called {\em algebraic Anosov flows} and play an important role in the theory. However, we now know many classes of Anosov flows beyond these constructions, thanks to the introduction of various surgery techniques since the late 1970s.

It is well established since the early 1980s that Anosov flows in dimension 3 have deep relation to geometry and topology, mostly thanks to the applications of {\em taut foliations} in low dimensional topology. That is, through the study of codimension 1 foliations $\mathcal{F}^u$ and $\mathcal{F}^s$, namely the {\em weak unstable and stable foliations}, which are tangent to the integrable invariant plane fields $E^{wu}:=E^u \oplus \langle X\rangle$ and $E^{ws}:=E^s \oplus \langle X\rangle$, namely the {\em weak unstable and stable bundles}, respectively. See \cite{bartintro} for a survey on some of these developments. It is important to know that while the splitting of $TM$ in the definition of Anosov flows is only Hölder continuous in general, subtle {\em regularity theory} of Anosov flows imply that the weak invariant bundles $E^{wu}$ and $E^{ws}$ are in fact $C^1$ \cite{invariant} and it is known that such regularity cannot be improved to $C^2$, except in the case of algebraic Anosov flows \cite{ghys}. Therefore, the foliation theoretic theory of Anosov 3-flows is often very topological in nature as the involved foliations are typically of low regularity.

It is convenient for our goals (contact geometric framework) in this paper to think of these foliations in terms of differential forms. More specifically, assuming the orientability of the plane field $E^{ws}$, we can describe it as the kernel of a non-vanishing $C^1$ 1-form $\alpha_u$ (with $\alpha_u\wedge d\alpha_u=0$ by the Frobenius theorem). Furthermore, such 1-form induces a norm on $E^u$ by $||e||=|\alpha_u(e)|$ for $e\in E^u$. Therefore, fixing the flow $X^t$, there is a 1-to-1 correspondence 
$$\bigg\{ \text{Non-vanishing }\alpha_u|\ker{\alpha_u}=E^{ws}\bigg\}/ \{ \pm 1\}
\mbox{\Large$\overset{\text{1-to-1}}{\longleftrightarrow}$}
\bigg\{ \text{Norms on } E^{u}\bigg\}.$$

Furthermore, considering the fact that $X$ preserves $\ker{\alpha_u}=E^{ws}$, we have 
$$\mathcal{L}_X\alpha_u=r_u\alpha_u,$$ 
for some function $r_u:M\rightarrow \mathbb{R}$, which measures the expansion of the norm $||.||$ corresponding to such $\alpha_u$, i.e. 
$$r_u \big|_p:=\partial_t \cdot \ln{||X_*^t (e_u)||}\bigg|_{t=0} \text{  for any }p\in M, e_u\in E_p^u.$$
 This is called the {\em expansion rate} of $E^u$. Furthermore, we observe that $r_u>0$, if and only if, the norm $||.||$ corresponding to $\alpha_u$ is adapted. Similarly, the expansion rate of $E^s$ with respect to some norm, denoted by $r_s$, can be defined via $\mathcal{L}_X\alpha_s=r_s\alpha_s$, where $\alpha_s$ is a non-vanishing 1-form with $\ker{\alpha_s}=E^{wu}$, as we have similar correspondence between norms on $E^s$ and such 1-forms. Again, corresponding to an adapted norm, we have $r_s<0$. Since adapted norms play an important role in this paper, we record this.

 \begin{definitionp}\label{normsad}
Give an Anosov flow $X^t$, we call a norm on $E^s\oplus E^u$ {\em adapted}, if it satisfies the rate condition in the definition with $A=1$. This is equivalent to the expansion rates $r_s,r_u$ induced from such norm satisfying $r_u>0>r_s$, assuming its differentiability along the flow.
 \end{definitionp}

 Strictly speaking $r_s<0<r_u$ is equivalent to the existence of a {\em hyperbolic splitting} $TM/\langle X\rangle\simeq E^s \oplus E^u$ (with exponentially expanding or contracting actions on $E^u$ and $E^s$, respectively) and then, the Doering's lemma implies that such splitting can be lifted to a splitting $TM\simeq \langle X \rangle \oplus E^s \oplus E^u$.

 These expansion rates can be described in various ways (see Section~2 of \cite{hoz6}) and in the following proposition, we quote some properties of these functions we will exploit in this paper.

\begin{proposition}\label{exppro}(\cite{hoz3,hoz6})
Let $r_u$ be the expansion rates corresponding to $\alpha_u$ and the flow generated by $X$. The following statements (and similar statements for the stable bundle $E^s$) hold:

(1) The expansion rate of the same norm with respect to $fX$ is $fr_u$.

(2) Generically, $\alpha_u$ and $r_u$ can be assumed as regular as $\ker{\alpha_u}=E^{ws}$ and in particular $C^1$.

(3) There is a correspondence between the norms on $E^{wu}/\langle X\rangle$ (which is independent of the parametrization) and the norm on the strong bundle $E^u$ of a given parametrization.

(4) The expansion rate corresponding to $\bar{\alpha}_u:=e^h\alpha_u$ can be computed as $\bar{r}_u=r_u+X\cdot h$. This implies that if both $r_u,\bar{r}_u>0$, the norms can be conjugated by a flow diffeomorphism.

(5) We have $||X^t_* (e_u)||/||e_u||=\exp{\int_0^t r_u(\tau)d\tau}$, where $r_u(\tau)=r_u \circ X^\tau$.
\end{proposition}

   \begin{figure}[h]
\centering
\begin{overpic}[width=0.5\textwidth]{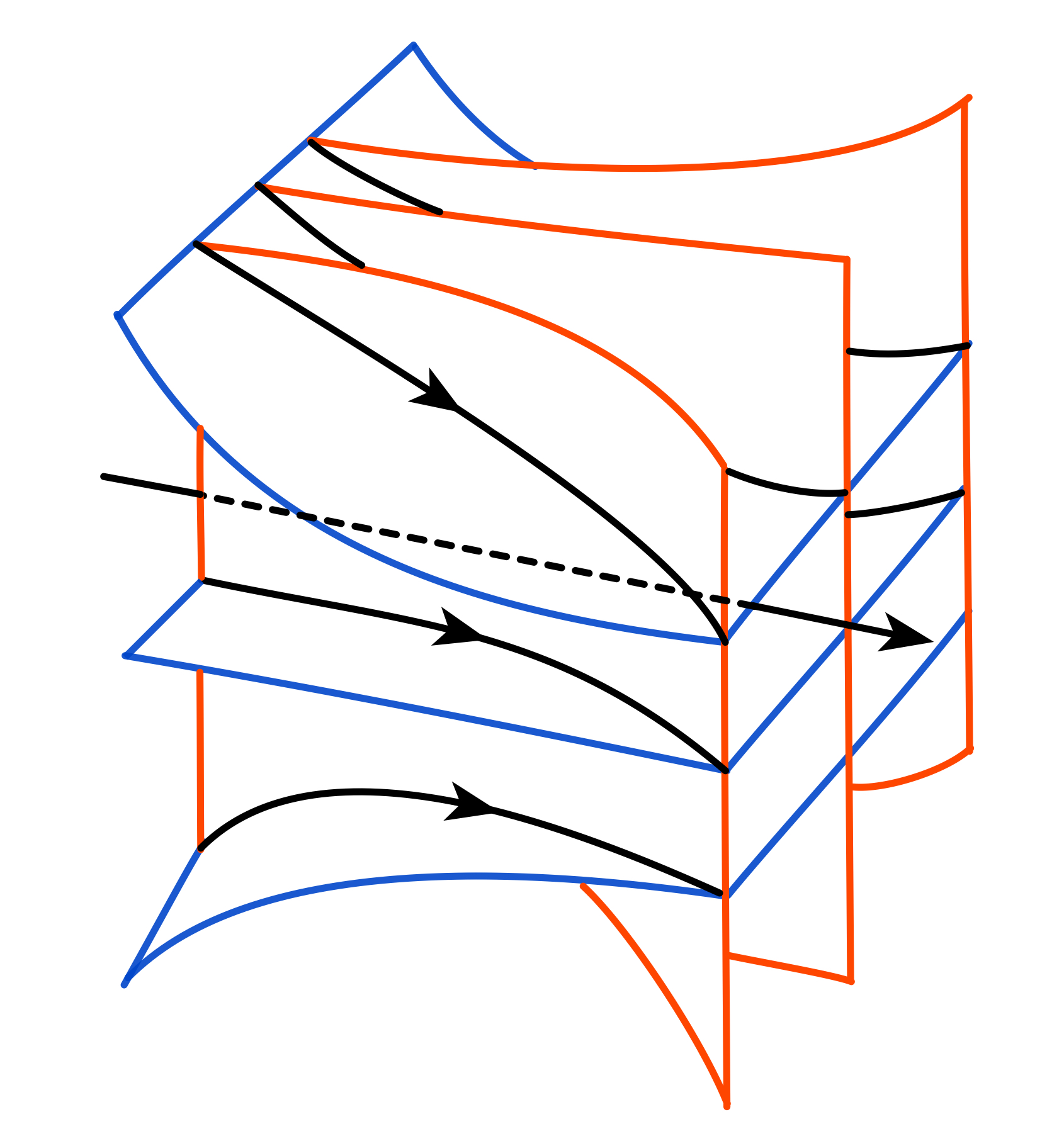}
       \put(170,15){$\mathcal{F}^u$}
              \put(08,35){$\mathcal{F}^s$}

          \put(08,145){$X$}
  \end{overpic}

\caption{Foliation theoretic local picture}
\end{figure}

These foliations and the existence of adapted norms yield a local picture for Anosov flows as shown in Figure~2. However, one should be careful interpreting this as a local model, as the (weak) invariant foliations are determined by the long term behavior of the flow and hence, the perturbation of the flow in a small neighborhood will result in global deformations of these foliations. Furthermore, the foliations in this picture have tangent fields with only $C^1$-regularity (with the exception of algebraic Anosov flows, where regularity is improved), and hence, Figure~2 is an honest picture at a low regularity level. However, the foliation theoretic viewpoint has been proven extremely successful in the development of topological theories for these structures.

\begin{convention}
We assume our Anosov flows to be {\em orientable}. That is, the invariant bundles are orientable. This is always the case, possibly after lifting to a double cover of $M$ (note that we are assuming $M$ to be oriented).
\end{convention}

In recent years, a new contact geometric approach has been under development, which hopes to reformulate and improve the classical theory based on foliations \cite{mit1,hoz3,sal1,clmm}. The goal of this paper is to take a step in this direction by providing a contact geometric local theory for these flows.

Recall that on an oriented 3-manifold $M$, a {\em contact form} $\alpha$ is a 1-form such that $\alpha \wedge d\alpha$ is non-vanishing. We call $\alpha$ a positive (negative) contact form, if compared to the orientation on $M$, $\alpha \wedge d\alpha$ is a positive (negative) volume form. The plane field $\xi=\ker{\alpha}$ is then called a (positive or negative) {\em contact structure}. Note that by the Frobenius theorem, a contact structure is a {\em maximally non-integrable} plane field (which is co orientable by our definition). This can be thought of as the extreme opposite of codimension 1 foliations thought as everywhere integrable plane fields. Moreover, a contact form defines a {\em co-orientation} on its kernel with respect to which $d\alpha|_{\ker{\alpha}}$ is a positive area form. One can easily define a positive (negative) contact form on $\mathbb{R}^3$ by $\alpha_{\pm,std}=dz\mp ydx$, named as the {\em standard contact forms} on $\mathbb{R}^3$. The classical {\em Darboux theorem} show that this example provides a local model for contact forms. That is, for any point $p\in M$ and positive contact form $\alpha_+$, there exists a diffeomorphism $\phi:U\rightarrow V$, where $U$ and $V$ are open neighborhoods of $p\in M$ and $O\in \mathbb{R}^3$, such that $\alpha_+=\phi^*\alpha_{+,std}$. The same statement holds for negative contact forms.

An important distinction between contact structures and foliations is that the contact condition $\alpha \wedge d\alpha$ is an open condition and hence, contact structures are {\em stable} structures (in fact, contact geometry's {\em Gray's theorem} implies that any deformation of contact structures is induced by an ambient isotopy). This is not the case for codimension 1 foliations, if interpreted as integrable plane fields. Furthermore, the openness of the contact condition often allows one to perturb the setting to have arbitrarily high regularity, which is again in contrast to foliations. This is particularly important for us, as our contact geometric theory of Anosov flows will enjoy the regularity as high as the flows, an important advantage compared to the foliation theoretic approach.

The study of the interactions between contact geometry and Anosov flows goes back to the following important observation made in the mid 1990s, independently by Mitsumatsu \cite{mit1} and Eliashberg-Thurston~\cite{et}.

\begin{lemma}
If $X^t$ is an Anosov flow on $M$. Then, there exists a transverse pair of negative and positive contact structures $\xi_-$ and $\xi_+$, such that $X\subset \xi_-\cap\xi_+$.
\end{lemma}

This can be shown easily by considering the flow action on the bi-sectors of the invariant bundles (or their smoothenings). The action of the flow $X_*^t$ would push such plane fields towards the unstable bundle $E^{wu}$ and away from $E^{ws}$. Since $X$ is included in these plane fields, the rotation exhibited by this action is simply a manifestation of their non-integrability, as one can argue using the Frobenius theorem. Based on the direction of the rotation, one of these contact structures is positive and the other is negative. See Figure~3 (left).

\begin{definition}\label{suppdef}
The pair $(\xi_-,\xi_+)$ of transverse negative and positive contact structures is called a {\em bi-contact structure}. We say the bi-contact structure $(\xi_-,\xi_+)$ {\em supports} a non-vanishing vector field $X$, if $X\subset \xi_- \cap \xi_+$. A positive (or negative) contact structure supports $X$, if it appears as in a bi-contact structure supporting $X$. A positive (or negative) contact form supports $X$, if its kernel supports $X$.
\end{definition}

Theorem~1.5 of \cite{hoz3} shows the space of these supporting (bi-)contact structures is contractible.

   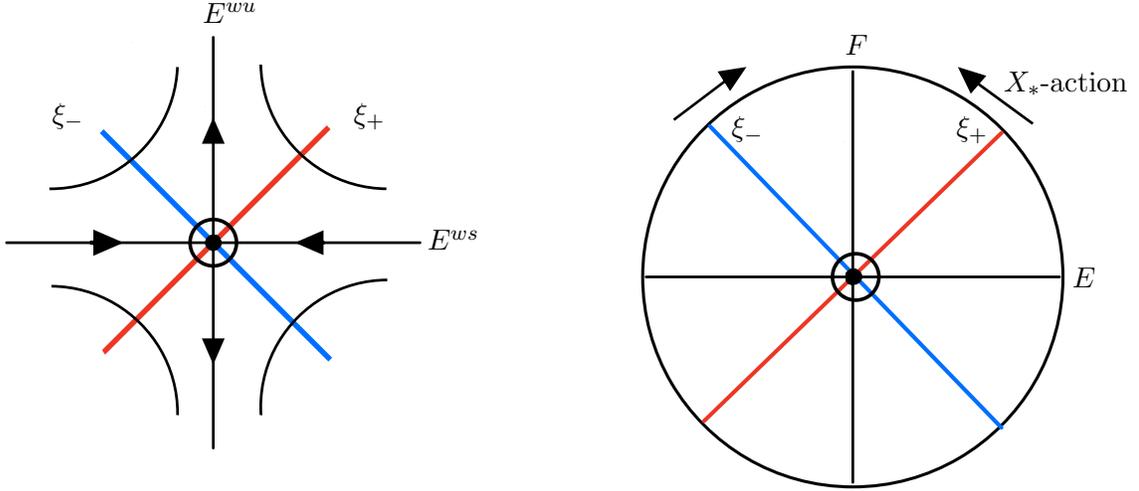
\begin{figure}[h]
\centering
\begin{overpic}[width=1\textwidth]{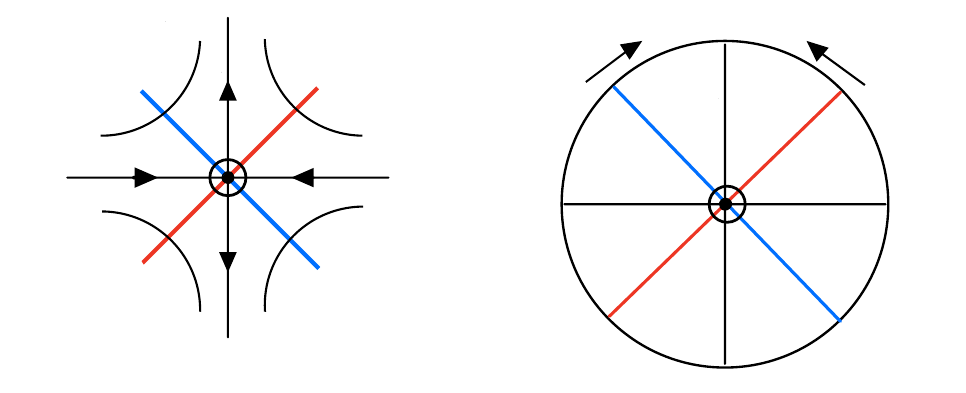}
        \put(164,150){$\xi_+$}
          \put(50,150){$\xi_-$}
        \put(192,102){$E^{ws}$}
        \put(107,188){$E^{wu}$}
        
                \put(436,88){$E$}
                  \put(350,176){$F$}
                   \put(410,162){$X_*$-action}
                   
                        \put(307,145){$\xi_-$}
          \put(392,145){$\xi_+$}
  \end{overpic}

\caption{Anosov flow action on $TM/\langle X\rangle$ (left) vs. projective Anosovity (right)}
\end{figure}

Even though the above observation is quite straightforward, it sits as the cornerstone for the contact geometric theory of Anosov 3-flows. First, we need to see that the contact geometric condition in the above lemma, i.e. $0\neq X\subset\xi_-\cap\xi_+$ for a transverse pair of negative and positive contact structures, has a dynamical interpretation for the non-vanishing vector field $X$.

Suppose $X$ is supported by a bi-contact structure $(\xi_-,\xi_+)$, i.e. $X\subset \xi_-\pitchfork \xi_+$ is non-vanishing. Studying the action of the flow $X_*^t$ on $\xi_-$ and $\xi_+$ and a standard contraction argument as in \cite{et} implies that there exist a unique transverse pair of plane fields $F=\lim_{t\rightarrow \infty}X_*^t\xi_-=\lim_{t\rightarrow \infty}X_*^t\xi_+$ and $E=\lim_{t\rightarrow -\infty}X_*^t\xi_-=\lim_{t\rightarrow -\infty}X_*^t\xi_+$, which are invariant under the flow $X^t$. This can be expressed as a continuous $X^t$-invariant splitting of the normal bundle $TM/\langle X\rangle \simeq E\oplus F$ and one can show that $E$ and $F$ are tangent to Hölder continuous {\em branching} foliations. See Figure~3 (right). However, the contact conditions for $\xi_-$ and $\xi_+$ imply more. Let $\alpha_+$ be {\em any} positive contact form with $\ker{\alpha_+}=\xi_+$. 

\begin{convention}\label{added}
Given the splitting $TM/\langle X\rangle \simeq E\oplus F$ as above, one can uniquely write $\alpha_+=\alpha_u-\alpha_s$, where $\alpha_u$ and $\alpha_s$ are non-vanishing 1-forms with $\ker{\alpha_u}=E$ and $\ker{\alpha_s}=F$, such that $\alpha_s \wedge \alpha_u$ is a positive area form on $TM/\langle X \rangle$. We keep this convention throughout this paper.
\end{convention}

 As in the Anosov case, $\alpha_u$ and $\alpha_s$ induce norms on $F$ and $E$, respectively. Therefore, we can define the expansion rates with respect to such induced norms $r_u$ and $r_s$. We call these norms or expansion rates {\em induced} from a supporting contact form (see more discussion on norms induced from supporting contact forms in \cite{hoz4}). One can then compute
$$\iota_X\alpha_+\wedge d\alpha_+=-\alpha_+\wedge \mathcal{L}_X\alpha_+=-(\alpha_u-\alpha_s)\wedge(r_u\alpha_u-r_s\alpha_s)=(r_u-r_s)\alpha_s\wedge \alpha_u.$$
Therefore, positive contacness of $\alpha_+$ is in fact equivalent to $r_u>r_s$. This means that with respect to the norms induced by $\alpha_+$, the (instantaneous) expansion of $F$ is always bigger than (instantaneous) expansion of $E$. This can be though of as a weak notion of hyperbolicity for flows, called {\em projective Anosovity}, where the flow action induces {\em relative expansion} of $F$ vs. $E$. More precisely, a non-vanishing vector field $X$ (or the flow $X^t$ generated by it) is called {\em projectively Anosov}, if it admits a continuous $X^t$-invariant splitting $TM/\langle X\rangle \simeq E\oplus F$, such that 
$$\frac{||X^t_*e_s||}{||X^t_* e_u||}\leq A e^{-Ct}\frac{||e_s||}{||e_u||}\text{ for any } e_s\in E \text{ and } e_u\in F,$$
for some constants $A,C>0$ and some norm $||.||$ on $TM/\langle X\rangle$. Such splitting is also called a {\em dominated splitting}, we say $F$ {\em dominates} $E$, or write $X_*|_E\prec X_*|_F$. 
One can again ask whether one can use a norm $||.||$ on $TM/\langle X\rangle$ with respect to which the relative expansion starts immediately. Gourmelon \cite{gour} provides an affirmative answer to this by further exploiting the averaging ideas of Holmes \cite{holmes} and Hirsch-Pugh-Shub \cite{invariant}.
We call a norm on $E\oplus F$ with this property (satisfying the rate condition of projective Anosovity with $A=1$) {\em projectively adapted}.
If we denote the induced expansion rates of $F$ and $E$ by $r_u$ and $r_s$, respectively, this is equivalent to $r_u>r_s$. Note that if for a projectively Anosov flow $X^t$, we have a norm on $E\oplus F$ with $r_u>0>r_s$, that would imply the hyperbolicity of the splitting $E\oplus F$, which as mentioned before can be lifted to an Anosov splitting, implying the Anosovity of $X^t$ with $E^{ws}=E$ and $E^{wu}=F$ (and hence our abuse of notation, when writing $\alpha_+=\alpha_u-\alpha_s$, and similarly, for the expansion rates $r_s,r_u$, etc. in the general projectively Anosov case. The statements of Proposition~\ref{exppro} still holds in this case).

The argument above shows that a vector field $X$ is projectively Anosov, if and only if, it is supported by a bi-contact structure. More specifically, any choice of {\em supporting} contact form $\alpha_+$ (or $\alpha_-$) yields a projectively adapted norm on $E\oplus F$.

We summarize this in the following theorem.

\begin{theorem}(Mitsumatsu 95 \cite{mit1}, Eliashberg-Thurston 95 \cite{et})
Let $X$ be a non-vanishing vector field on $M$. TFAE:

(1) $X$ generates a projectively Anosov flow;

(2) $X$ is supported by a bi-contact structure $(\xi_-,\xi_+)$.
\end{theorem}

\begin{convention}
Since stable and unstable bundles of Anosov flows are defined independently of the manifold's orientations, as in the above computations, we need to adopt a convention in our computations. In this paper and in the Anosov case, we orient the invariant bundles by the choice of vector fields $e_s\in E^s,e_u\in E^u$, or equivalently, by the choice of non-vanishing foliation forms $\alpha_s,\alpha_u$,  such that $(X,e_s,e_u)$ is an oriented basis for $TM$, or equivalently, the volume form defined by $\iota_X \Omega =\alpha_s\wedge \alpha_u$ is positive. This orientation convention can be naturally extended to general projectively Anosov flows, by replacing $E^s$ and $E^u$ with $E$ and $F$, respectively (and lifting the vectors from $TM/\langle X\rangle$ to $TM$ if needed).
\end{convention}

    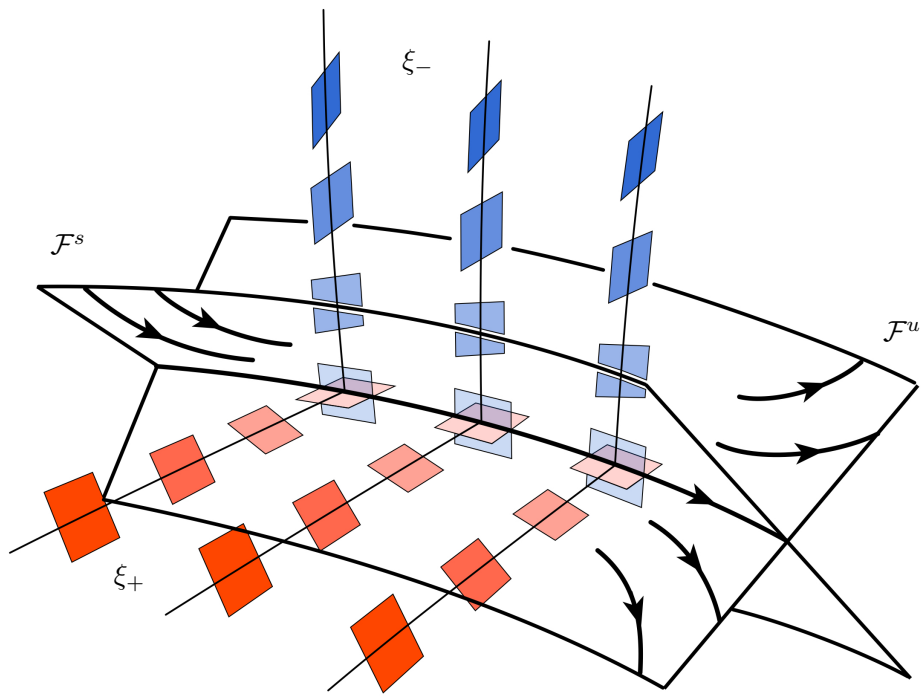
\begin{figure}[h]
\centering
\begin{overpic}[width=0.8\textwidth]{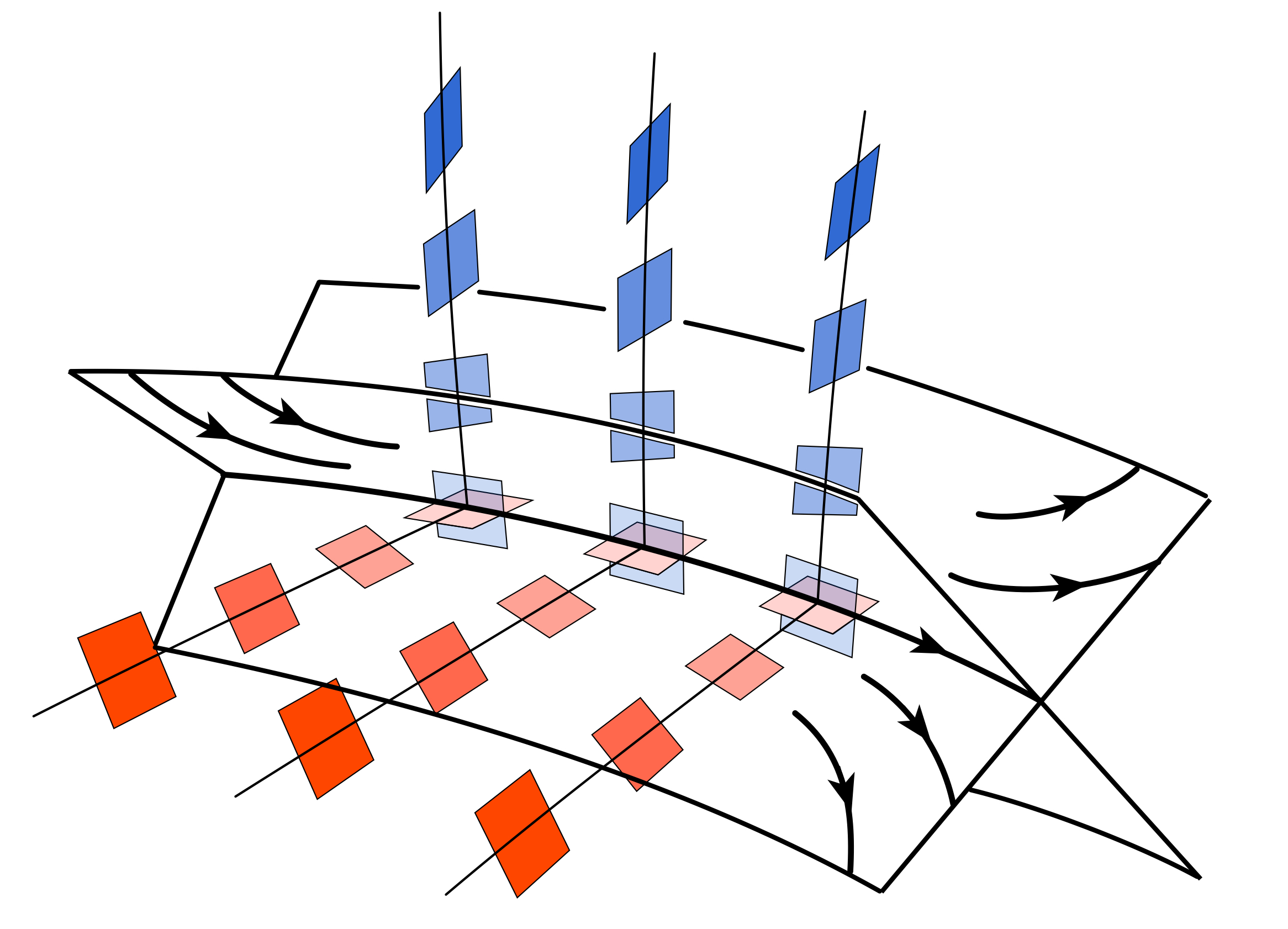}

       \put(158,245){$\xi_-$}
         \put(49,50){$\xi_+$}
          \put(25,175){$\mathcal{F}^s$}

                        \put(340,142){$\mathcal{F}^u$}
  \end{overpic}

\caption{Bi-contact geometric approach to Anosov 3-flows}
\end{figure}

The importance of this theorem is that it provides a purely contact geometric characterization of a dynamical condition, i.e. {\em projective Anosovity}, and therefore can be used as a bridge between different aspects of contact geometry and hyperbolic dynamics in a broad sense, and particularly, Anosov dynamics. One can naturally ask whether a similar characterization is possible for Anosov flows, i.e. to give a purely contact geometric characterization of Anosov flows. This has been carried out in \cite{hoz3} via connections to 4-dimensional Liouville (and symplectic) geometry, providing a 4-dimensional geometric theory of Anosov flows and recently resulting in the introduction of various new symplectic geometric invariants of Anosov flows by Cieliebak-Lazarev-Massoni-Moreno \cite{clmm}. The goal of this paper is provide an alternative 3 dimensional theory based on interactions with the Reeb dynamics associated with a supporting bi-contact structure, opening new avenues in exploring the interactions between Anosov dynamics and the vast world of contact geometry.

\section{Reeb dynamics of bi-contact structures and strong adaptations}\label{s3}

In this section, we study the Reeb flows associated with a bi-contact structure supporting a (projectively) Anosov 3-flow. This will eventually result in the introduction of our local model.

We start with an elementary observation first noted in \cite{hoz3}.

\begin{lemma}\label{reebquad}
Let $(\xi_-,\ker{\alpha_+})$ be a bi-contact structure supporting a projectively Anosov flow $X^t$. Then,
$$R_{\alpha_+}\subset \langle X,r_ue_s+r_se_u\rangle,$$
where $e_s,e_u$ and $r_s,r_u$ are the unit vectors and expansion rates, respectively, with respect to the norm induced from $\alpha_+$, and $(X,e_s,e_u)$ is an oriented basis on $M$.
\end{lemma}

\begin{proof}
Write $\alpha_+=\alpha_u-\alpha_s$ and let $r_u ,r_s$ be the induced expansion rates. We have 
$$\iota_{R_{\alpha_+}}d\alpha_+=0 \Longrightarrow 0=d\alpha_+(X,R_{\alpha_+})=(\mathcal{L}_X\alpha_+)(R_{\alpha_+})=(r_u\alpha_u-r_s\alpha_s)(R_{\alpha_+})$$
$$\Longleftrightarrow R_{\alpha}\subset \langle X,r_ue_s+r_se_u\rangle,$$
where $e_u,e_s$ are the unit vectors with respect to the induced norm and $(e_s,e_u,X)$ is an oriented basis.
\end{proof}

An immediate consequence of the above observation is that the information about the Anosovity of the flow, i.e. having $r_s<0<r_u$ with respect to some adapted norm, can be described in terms of which quadrants $\langle R_{\alpha_+},X \rangle$ occupies in $TM/\langle X\rangle \simeq E\oplus F$. See Figure~5 (left). In \cite{hoz3}, we call $R_{\alpha_+}$ {\em dynamically negative} if it lies in the second or forth quadrant of Figure~5 (left). Our computations indicate that this is equivalent to $r_s<0<r_u$, i.e. the norms induced from $\alpha_+$ on $E^s$ and $E^u$ are adapted in the sense of Anosov flows. We naturally have the following definition.

\begin{definitionp}\label{cad}
Let $X^t$ be a projectively Anosov flow. We call a supporting contact form $\alpha_+$ {\em adapted}, if $R_{\alpha_+}$ lies in the second or forth quadrant of Figure~5 (left). That is equivalent to the norms induced from $\alpha_+$ on the invariant bundles being adapted in the sense of Definition/Proposition~\ref{normsad}.
\end{definitionp}

   \begin{figure}[h]
\centering
\begin{overpic}[width=1\textwidth]{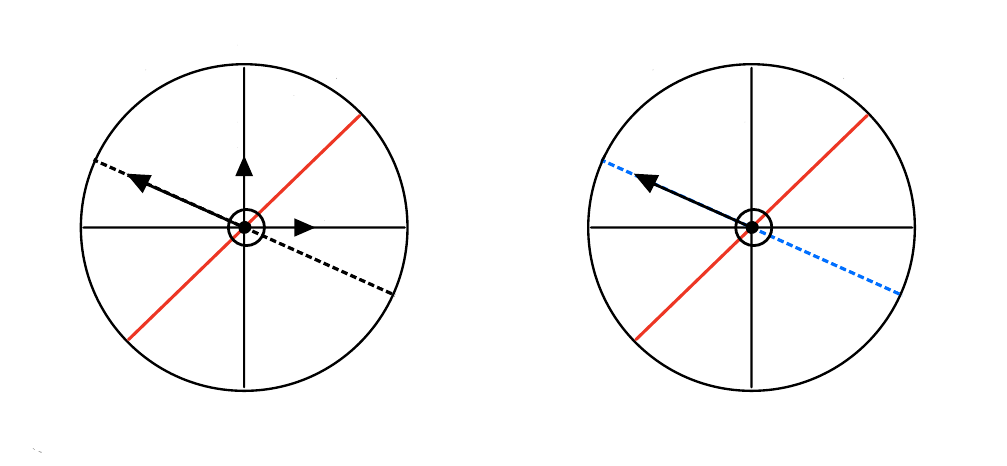}
        \put(172,165){$\ker{\alpha_+}$}
          \put(60,140){$R_{\alpha_+}$}
            \put(185,70){$\langle X, r_ue_s+r_se_u\rangle$}
        \put(192,108){$E^{ws}$}
        \put(110,192){$E^{wu}$}

        \put(408,165){$\ker{\alpha_+}$}
          \put(296,140){$R_{\alpha_+}$}
            \put(421,70){$\xi_-$}
        \put(428,108){$E^{ws}$}
        \put(346,192){$E^{wu}$}
  \end{overpic}

\caption{Projective geometry of Reeb flows in bi-contact geometry (left) Strong adaptation (right)}
\end{figure}
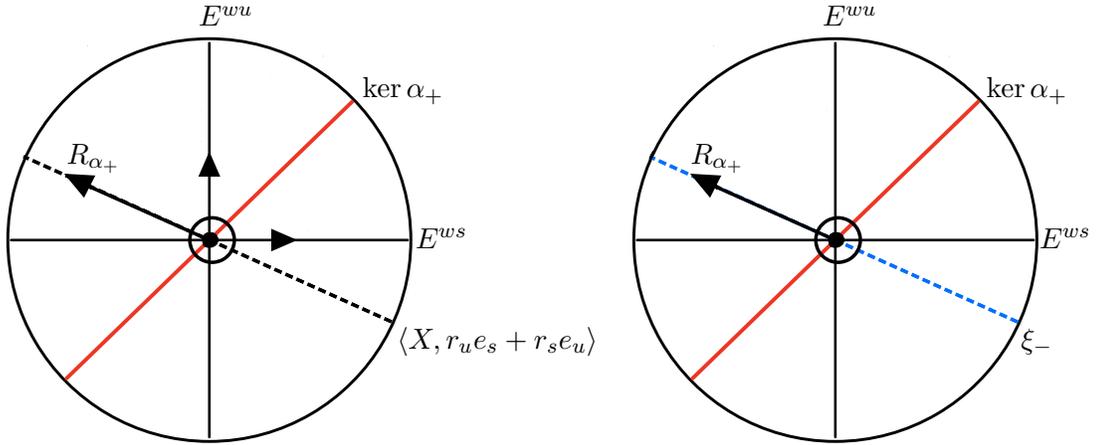

Our argument shows the following.

\begin{corollary}(H. 20 \cite{hoz3})
Let $X$ be a non-vanishing vector field on $M$. TFAE:

(1) $X$ generates an Anosov flow;

(2) $X$ is supported by an adapted contact form $\alpha_+$.
\end{corollary}

One can further show that the space of adapted contact forms is a convex subset of supporting contact forms, thanks to the convexity of the space of adapted norms (see Lemma~\ref{adnormconvex}).

As discussed in Section~6 of \cite{hoz3}, there are interesting consequences of the above, including the {\em hypertightness} of contact structures supporting Anosov flows, or a symplectic generalization of non-existence of Anosov flows on $\mathbb{S}^3$. However, this is not really a purely contact geometric condition, as referring to the invariant splitting $E\oplus F$ is required, in order to determine whether a contact form is adapted. The next observation attempts to bypass this disadvantage by considering an additional condition.

\begin{corollary}\label{strcor}
If $(\xi_-,\ker{\alpha_+})$ is a bi-contact structure supporting a projectively Anosov flow $X^t$ and $R_{\alpha_+}\subset \xi_-$, then $X^t$ is Anosov.
\end{corollary}

Notice that the condition $R_{\alpha_+}\subset \xi_-$ is (a priori strictly) stronger than $\alpha_+$ being an adapted contact form (as any supporting $\xi_-$ lies in the second/forth quadrant of Figure~5 (left)) and hence, the above conclusion is straightforward. 
As we will soon see, it is not hard to prove the existence of such contact forms in the case of incompressible Anosov flows.


It is important to notice that the condition of Corollary~\ref{strcor} can be viewed as a condition on $\alpha_+$. More precisely, this is equivalent to $\langle X,R_{\alpha_+}\rangle$ being a negative contact structure (and then we can let $\xi_-:=\langle X,R_{\alpha_+}\rangle$ to have the desired bi-contact structure). In this sense, it is a $C^2$-open condition on $\alpha_+$. 

\begin{definition}
Let $X$ be any non-vanishing vector field. We call a supporting contact form $\alpha_+$ {\em strongly adapted}, if $\langle X,R_{\alpha_+}\rangle$ is a negative contact structure.
\end{definition}

Our argument above can be summarized by saying that a vector field supported by a strongly adapted contact form is Anosov. The main theorem of this paper (Theorem~\ref{main}) establishes an inverse for Corollary~\ref{strcor}, establishing a contact geometric characterization of Anosov 3-flows. Also, notice that this definition is well defined for Anosov vector fields up to (unoriented) reparametrizations.

\begin{remark}
In general, the supporting (bi-)contact geometry of (projective) Anosov flows is independent of the (unoriented) reparametrization classes of Anosov flows, as reversing the orientation of the flow does not affect the role of positive and negative contact structure, while it reverses the roles of the stable and unstable bundles. On the other hand, reversing the orientation of the ambient manifold reverses the roles of positive and negative contact structures, while leaving the roles of stable and unstable bundles unaffected.

In particular, the definition of (strongly) adapted contact forms is independent of (unoriented) parametrizations. On the other hand, a norm on $E^u$ or $E^s$ being (strongly) adapted depends on the specific parametrization of the flow (even in the same orientation class).
\end{remark}

To understand this condition better, we start with the following computations. Fix an Anosov flow $X^t$ and an adapted contact form $\alpha_+$. Using Lemma~\ref{reebquad} and the notations as before, we have

$$\alpha_+\text{ strongly adapted}\Longleftrightarrow \langle X,r_ue_s+r_se_u\rangle \text{ is a negative contact structure}$$
$$\Longleftrightarrow r_u\alpha_u-r_s\alpha_s=\mathcal{L}_X\alpha_+ \text{ is a negative contact form}$$
$$ \Longleftrightarrow \text{the norms induced from }\mathcal{L}_X\alpha_+ \text{ are projectively adapted}$$
$$[r_u+X\cdot \ln{r_u}]-[r_s + X\cdot \ln{(-r_s)}]>0.$$

Here, $X\cdot$ refers to the directional derivative of a function in the direction of $X$. Notice that the norm induced from $\mathcal{L}_X\alpha_+$ on $E^u$ corresponds to the non-vanishing 1-form $r_u\alpha_u$ and therefore, has an expansion rate of $[r_u+X\cdot \ln{r_u}]$. Similarly, the norm induced from $\mathcal{L}_X\alpha_+$ on $E^s$ corresponds to the non-vanishing 1-form $-r_s\alpha_s$ and therefore, has an expansion rate of $[r_s+X\cdot \ln{(-r_s)}]$.

In the case $X^t$ is an incompressible Anosov flow, one can define a supporting contact form $\alpha_+=\alpha_u-\alpha_s$ with its inducing expansion rates satisfying $r_u=-r_s$ and in particular, after a reparametrization of $X$, we can assume $r_u=-r_s=1$. We have $\mathcal{L}_X\alpha_+=\alpha_u+\alpha_s$ is a negative contact form (since it's induced norms on $E^s\oplus E^u$ is adapted). Note that the contact form $\alpha_+$ in the above construction is only $C^1$ (with $C^1$ derivative along the flow), unless $X^t$ is algebraic Anosov, in which case, it is $C^\infty$. However, we can always perturb such $\alpha_+$ to obtain a $C^\infty$ strongly adapted contact form, thanks to the openness of the contact condition.

It turns out that this idea can be improved to give a purely contact geometric characterization of incompressible Anosov flows. As discussed in \cite{hoz4}, the symmetry of an existing volume form is manifested in the strongly adapted contact geometry of Anosov flows.

\begin{theorem}(H. 21 \cite{hoz4})\label{vol}
Let $X$ be a non-vanishing vector field on $M$. TFAE:

(1) $X$ generates an incompressible Anosov flow;

(2) $X$ is supported by a bi-contact structure $(\ker{\alpha_-},\ker{\alpha_+})$ such that $\alpha_+(R_{\alpha_-})=\alpha_-(R_{\alpha_+})=0$.
\end{theorem}

\begin{remark}\label{biadaptedrem}
The way that the proof Theorem~\ref{vol} is phrased in \cite{hoz4} yields this result in the $C^1$-category (in particular, the constructed contact forms as in (2) are a priori only as regular as the weak invariant bundles), with some improvements on partial regularity (see Remark~5.2 in \cite{hoz4}). But here, we would like to argue that this result in fact holds in the $C^\infty$ category.

Let $\alpha_+$ be a $C^\infty$ strongly adapted contact form for $X$. There exists a $C^\infty$ positive function $h$ such that $(\mathcal{L}_X \alpha_+) \wedge \alpha_+=h\iota_X\Omega$, where $\Omega$ is an invariant volume form, which is $C^\infty$ by an application of the {\em Livsić theorem} \cite{delallave}. Note that unlike the invariant volume form $\Omega$, the invariant transverse 2-form $\beta:=\iota_X\Omega$ is independent of reparametrizations of $X$. Now consider the $C^\infty$ reparametrization $\bar{X}:=\frac{1}{h}X$, which yields $(\mathcal{L}_{\bar{X}} \alpha_+) \wedge \alpha_+=\beta$. Define $\alpha_-:=\mathcal{L}_{\bar{X}}\alpha_+$ which is a $C^\infty$ supporting negative contact form by construction. Note that we have $\alpha_+(R_{\alpha_-})=0$, if and only if, $(\mathcal{L}_{\bar{X}}\alpha_-)\wedge \alpha_+=0$. The latter can be easily observed since $0=\mathcal{L}_{\bar{X}}\beta=(\mathcal{L}_{\bar{X}}\mathcal{L}_{\bar{X}} \alpha_+) \wedge \alpha_+=(\mathcal{L}_{\bar{X}} \alpha_-) \wedge \alpha_+$. Hence, $\alpha_-$ and $\alpha_+$ are $C^\infty$ contact forms satisfying the conditions of Theorem~\ref{vol} for an arbitrary incompressible Anosov 3-flow, even though the weak invariant bundles are only $C^1$ in general.

\end{remark}

The computation above show that a contact form $\alpha_+$ is strongly adapted when the functions $X\cdot r_s,\  X\cdot r_u$ are sufficiently small. Notice that $r_u-r_s>0$ by the contactness of $\alpha_+$. This can be achieved by a technical theorem discussed in the next section, namely the {\em asymptotic synchronization theorem}, which exploits and refines the very same averaging techniques of Holmes \cite{holmes} and Hirsch-Pugh-Shub \cite{invariant} which was used to show the existence of adapted norms. But first, it is useful to observe various formulations for a contact form being strongly adapted to an Anosov 3-flow. Recall that this is well defined up to (unoriented) scaling of the vector field $X$.

\begin{proposition}\label{strad}
Fix an Anosov flow $X^t$ and suppose $\alpha_+$ is adapted to $X$. TFAE:

(1) $\alpha_+$ is strongly adapted to $X$;

(2) $(\langle X,R_{\alpha_+}\rangle,\ker{\alpha_+})$ is a supporting bi-contact structure for $X$;

(3) $\mathcal{L}_X\alpha_+$ is a negative contact form;

(4) the norm $\mathcal{L}_X \alpha_+$ induces on $E^s\oplus E^u$ is projectively adapted;

(5) $(X,[R_{\alpha_+},X])$ is a co-oriented basis for $\ker{\alpha_+}$;

(6) $(R_{\alpha_+},[X,R_{\alpha_+}],X)$ is an oriented basis for $TM$;

(7) $\mathcal{L}_X \alpha_+ ([R_{\alpha_+},X])>0$;

(8) $[r_u+X\cdot \ln{r_u}]>[r_s + X\cdot \ln{(-r_s)}]$, where $r_u,r_s$ are expansion rates induced from $\alpha_+$.
\end{proposition}

\begin{proof}
The equivalence of $(1)\Leftrightarrow (2) \Leftrightarrow (3)\Leftrightarrow (8)$ is established in the earlier computations. To see the equivalence of $(4)$, $(5)$ and $(6)$ with these, notice that $[X,R_{\alpha_+}]\subset\ker{\alpha_+}$ for any contact form $\alpha_+$ and any Legendrian vector field $X\subset \ker{\alpha_+}$, since
$$\alpha[X,R_{\alpha_+}]=d\alpha_+(R_{\alpha_+},X)+X\cdot \alpha_+(R_{\alpha_+})-R_{\alpha_+} \cdot {\alpha_+}(X)=0.$$
But in our case, it is easy to check that $(R_{\alpha_+},X)$ is a co-oriented basis for the negative contact structure $\ker{\mathcal{L}_X\alpha_+}=\langle X,R_{\alpha_+}\rangle$, which means
$$0>d\mathcal{L}_X\alpha_+ (R_{\alpha_+},X)=\mathcal{L}_X\alpha_+[X,R_{\alpha_+}],$$
which, keeping track of orientations, one can observe that it is equivalent to $([X,R_{\alpha_+}],X)$ being co-oriented basis for $\ker{\alpha_+}$. Note that this is equivalent to $(R_{\alpha_+},[X,R_{\alpha_+}],X)$ being an oriented basis for $TM$.
\end{proof}

Modulo the asymptotic synchronization theorem, we are ready to present the contact geometric characterization of Anosov 3-flows in terms of (strongly adapted) contact geometry.

 \begin{theorem}\label{main}
 Let $X$ be a non-vanishing vector field on $M$. TFAE:
 
 (1) $X$ generates an Anosov flow;
 
 (2) $X$ is supported by a {\em strongly adapted} contact form $\alpha_+$.
 \end{theorem}

 \begin{proof}

 $(2)\Rightarrow (1)$ is discussed above in Corollary~\ref{strcor} and is a consequence of \cite{hoz3} (Section 6).
 
  For $(1)\Rightarrow (2)$, let $\alpha_+$ be any adapted contact form for $X$ and write  $\alpha_+=\alpha_u-\alpha_s$, inducing expansion rates $r_{u}$ from $\alpha_{u}$ (see Convention~\ref{added}). Reparametrize $X$ such that $r_s\equiv -1$. By Proposition~\ref{strad}, we have 
  $$\alpha_+ \text{ strongly adapted}\Longleftrightarrow [r_u+X\cdot \ln{r_u}]-[r_s+X\cdot \ln{(-r_s)}]=1+r_u+X\cdot \ln{r_u}>0.$$
  Since $r_u>0$, this means that having $|X\cdot \ln{r_u}|<1$ is enough to guarantee that $\alpha_+$ is strongly adapted. The technical part of this proof is to show that one can deform such $\alpha_u$ to satisfy this condition. This is possible thanks to the {\em asymptotic synchronization process} discussed in Section~\ref{s4}. More precisely, it is enough to prove the following lemma, whose technical core is left to be proved in the next section.
  
  \begin{lemma}\label{tech}
  Given $\alpha_u$ as above, there exists a family of 1-forms $\alpha_{u,T}$ for $T\geq 0$ such that $\ker{\alpha_{u,T}}=\ker{\alpha_u}$, $\alpha_{u,0}=\alpha_u$ and
  $$[1+r_{u,T}+X\cdot \ln{r_{u,T}}] -[ 1+\frac{1}{T}\int^T_0 r_u(\tau) d\tau] \mbox{\Large$\overset{\text{unif.}}{\longrightarrow}$} 0 \ \text{ as }T \longrightarrow \infty,$$
 where $r_{u,T}$ is the expansion rate induced from $\alpha_{u,T}$ and $\mbox{\Large$\overset{\text{unif.}}{\longrightarrow}$}$ denotes the uniform convergence.
  \end{lemma}
  
  \begin{proof}
  The idea of the proof is to take an appropriate average of the norm induced by $\alpha_u$, using the flow generated by $X$. This has been carried out carefully in the next section and exploiting an averaging technique of Holmes \cite{holmes} and Hirsch-Pugh-Shub \cite{invariant}. More specifically, we construct in Theorem~\ref{expunif},
   a family of adapted 1-forms $\alpha_{u,T}$ such that 
   $$\begin{cases}
 r_{u,T} - \frac{1}{T}\int_0^Tr_u (\tau)d\tau \mbox{\Large$\overset{\text{unif.}}{\longrightarrow}$} 0 \\
  X\cdot r_{u,T}\mbox{\Large$\overset{\text{unif.}}{\longrightarrow}$}  0\end{cases}\text{ as }T \longrightarrow \infty, $$
   where $ r_{u,T}$ is the expansion rate induced from $\alpha_{u,T}$. Notice that since $\frac{1}{T}\int_0^Tr_u (\tau)d\tau\geq\min{r_u}>0$, this implies
  $$  X\cdot \ln{r_{u,T}}=\frac{ X\cdot r_{u,T}}{r_{u,T}}\ \mbox{\Large$\overset{\text{unif.}}{\longrightarrow}$}   0\ \text{ as }T \longrightarrow \infty,$$
  which completes the proof of the lemma.
  \end{proof}
  
Now, using the family of 1-forms $\alpha_{u,T}$ given in Lemma~\ref{tech}, one can define an (in general only $C^1$) family of contact forms $\alpha_{+,T}:=\alpha_{u,T}-\alpha_s$. Furthermore, there exists some $T_0 \geq 0$, where $\alpha_{+,T}$ is strongly adapted for any $T>T_0$. That is thanks to the fact that $$1+\frac{1}{T}\int^T_0 r_u(\tau) d\tau\geq 1+\min{r_u}>0.$$
Finally, for any $T>T_0$, one can approximate $\alpha_{u,T}$ with a $C^\infty$ supporting contact form $\bar{\alpha}_{u,T}$ such that the resulting expansion rate $\bar{r}_{u,T}$ and the derivative $X\cdot \bar{r}_{u,T}$ are arbitrarily close to $r_{u,T}$ and $X\cdot r_{u,T}$, respectively (see Section~3.6 of \cite{hoz6}). Therefore, the $C^\infty$ contact forms $\bar{\alpha}_{u,T}$ can be chosen to be strongly adapted as well. 

 \end{proof}

       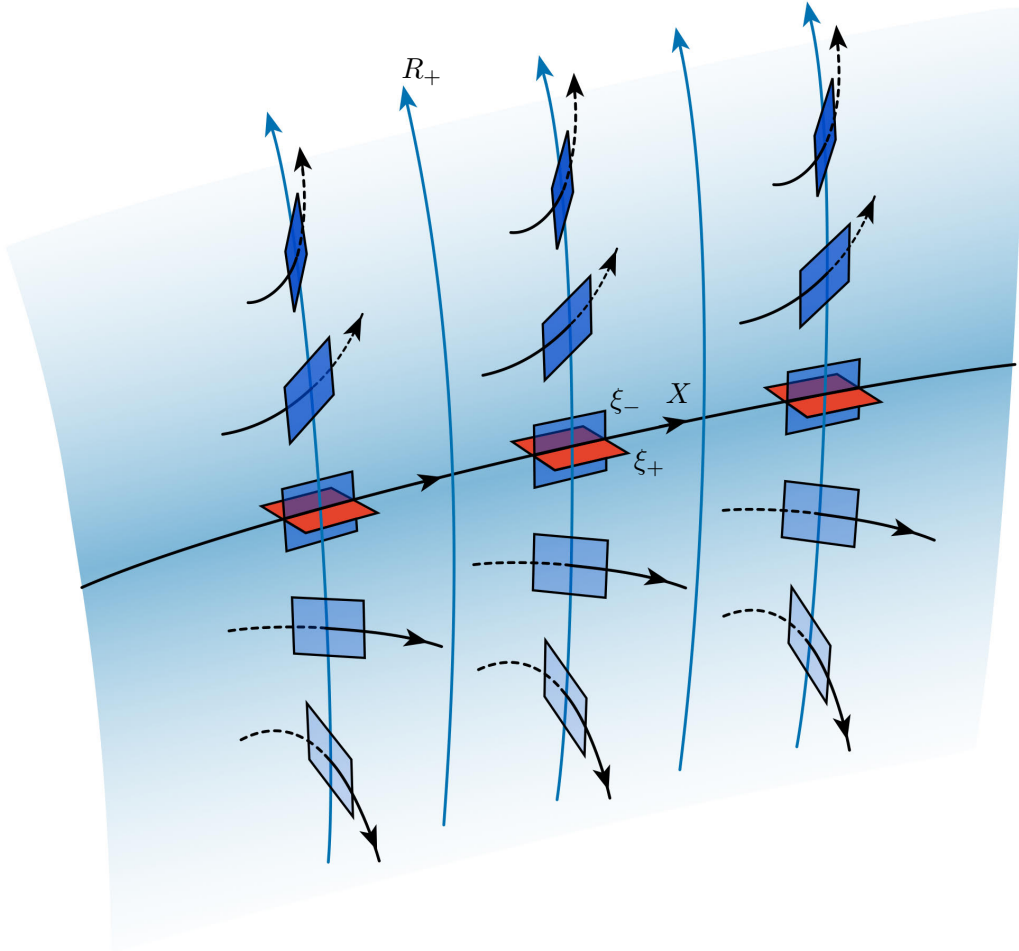
\begin{figure}[h]
\centering
\begin{overpic}[width=0.85\textwidth]{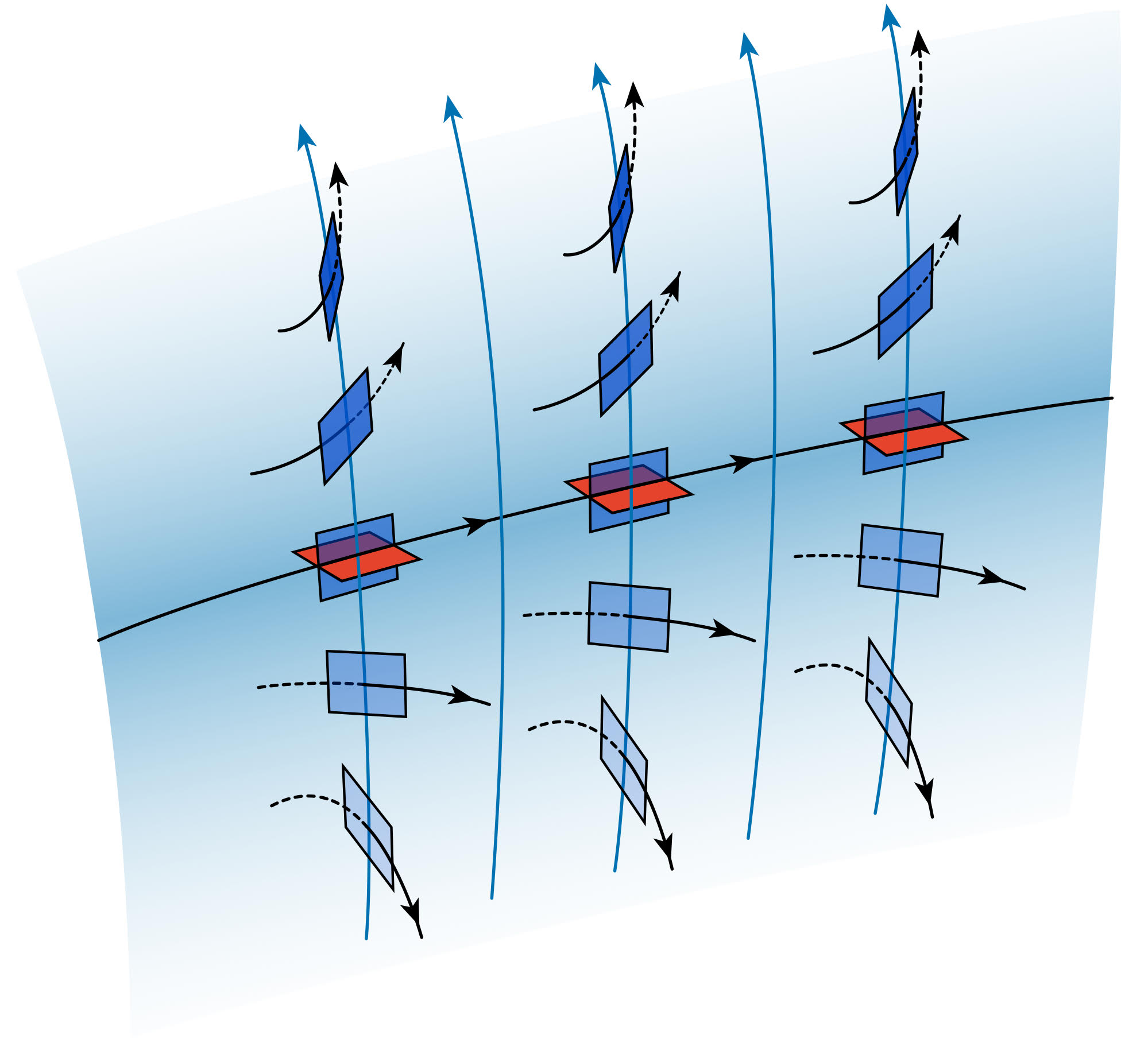}
        \put(243,190){$\xi_+$}
          \put(234,213){$\xi_-$}
        \put(255,215){$X$}
        \put(155,338){$R_+$}

  \end{overpic}

\caption{Local model for strongly adapted contact geometry}
\end{figure}

 
 Finally, we would like to remark that the contact geometric model provided in Theorem~\ref{main} describes the {\em minimal} purely contact geometric conditions guaranteeing the Anosovity of a Legendrian flow. However, the convergence obtained in the asymptotic synchronization Theorem~\ref{expunif} is much stronger and indeed allows one to {\em asymptotically synchronize} this contact geometric model along the flow. 
 
 More precisely, recall that $\alpha_+$ being strongly adapted to $X$ is equivalent to the norms induced from $\alpha_+$ and $\mathcal{L}_X \alpha_+$ being adapted and projectively adapted, respectively (see Proposition~\ref{strad}). One can strengthen this condition by requiring the norms induced from both $\alpha_+$ and $\mathcal{L}_X \alpha_+$ to be adapted. It is easily seen that is equivalent to requiring
 $$[r_s+X\cdot \ln{(-r_s)}]<0<[r_u+X\cdot \ln{r_u}],$$
 which is a condition stronger than strong adaptation of $\alpha_+$ (see Proposition~\ref{strad}).
 
We naturally call a norm on $E^u$ (or $E^s$) with $[r_u+X\cdot \ln{r_u}]>0$ (or $[r_s+X\cdot \ln{(-r_s)}]<0$) {\em strongly adapted} (see Definition/Proposition~\ref{snorm} for a more precise and general definition) and note that this condition depends on the parametrization (and is not invariant under scaling of the generating vector field $X$). Therefore, the condition on such $\alpha_+$ can be described by stating that the norms induced from $\alpha_+$ are strongly adapted.
 The asymptotic synchronization process of Theorem~\ref{expunif} similarly provides a contact form with this property as $X\cdot \ln{(-r_s)}$ and $X\cdot r_u$ can be taken uniformly small. So, Similar to Theorem~\ref{main}, we can have the following (see Corollary~\ref{strnormex}).
 
 \begin{theorem}\label{mains}
 Let $X$ be a non-vanishing vector field on $M$. TFAE:
 
 (1) $X$ generates an Anosov flow;
 
 (2) $X$ is supported by a contact form $\alpha_+$ inducing strongly adapted norms on $E^u$ and $E^s$.
 \end{theorem}
 
 The benefit of considering such sub class of adapted contact forms is that it is convex (while fixing $X $), which can be useful in computations (see Lemma~\ref{adnormconvex}). This is a refinement of the strong adapted contact geometry which can be pushed further.
 
 \begin{remark}\label{strongfilter}
 We can further refine our contact geometric model by imposing further assumptions, resulting in the following filtration of the space of strongly adapted contact forms given for an Anosov flows $X^t$:
 
 $$...\ \mbox{\Large$\subset$}\bigg\{\substack{\text{Strongly adapted $\alpha_+$} \\  \text{strongly adapted $\mathcal{L}_X\alpha_+$} \\ \text{Adapted $\mathcal{L}_X\mathcal{L}_X \alpha_+$}} \bigg\} \mbox{\Large$\subset$}
 \bigg\{\substack{\text{Strongly adapted $\alpha_+$} \\  \text{Strongly adapted $\mathcal{L}_X\alpha_+$}} \bigg\} \mbox{\Large$\subset$}
 \bigg\{\substack{\text{Strongly adapted $\alpha_+$} \\  \text{Adapted $\mathcal{L}_X\alpha_+$}} \bigg\} \mbox{\Large$\subset$}
\bigg\{\substack{\text{Strongly adapted $\alpha_+$} \\  \text{}} \bigg\}.$$
Notice that all the classes defined only in terms of strong adaptation (and not adaptation) can be described in purely contact geometric terms. Theorem~\ref{expunif} implies that all these sets are in fact non-empty for an Anosov flow. Therefore, given an Anosov flow $X^t$, one can equip it with an adapted contact form {\em with an arbitrary degree of strong adaptations}. Conversely, it is obvious that for a projectively Anosov flow, any of the above sets being non-empty implies Anosovity. Therefore, further refinements the characterizations given in Theorem~\ref{main} and \ref{mains} are possible.
\end{remark}


\section{Asymptotic synchronization process}\label{s4}

This section is devoted to proving the asymptotic synchronization theorem, which is a technical ingredient needed in the proof of Theorem~\ref{main}. We discuss our result in a more general setting for the sake of future reference, as we believe the applications of this theorem to go beyond what is proven in this paper. As a matter of fact, this theorem is already used in \cite{hoz5} to study and resolve an old problem in geometric analysis of contact metrics, i.e. the {\em Chern-Hamilton problem}.

 In this section, we let $\pi:E\rightarrow M$ be a line bundle over a (not necessarily closed) manifold $M$ of arbitrary dimension and $X^t$ be a flow on $M$, inducing a bundle action $X_*$ on $E$, which is $N$-times differentiable along the flow, where $1\leq N\in\mathbb{N \cup \{\infty\}}$. More precisely, we are assuming the existence of a family of fiberwise linear and non-trivial maps $X_*^t:E\rightarrow E$ such that the diagram
 \[\begin{tikzcd}
E \arrow{r}{X^t_*} \arrow[swap]{d}{\pi} & E \arrow{d}{\pi} \\
M \arrow{r}{X^t} & M
	\end{tikzcd}\]
 commutes for any $t\in \mathbb{R}$, where $X_*^t$ is $N$ times differentiable with respect to $t$. We further assume that the fiberwise norm $||.||$ on $E\rightarrow M$ is $N$-time differentiable under the action of $X^t_*$, i.e. $||X^ t_*||$ is $N$-times differentiable with respect to $t$.
 
 The following lemma shows that we can deform a given norm on $E$ in order to make {\em arbitrarily uniform} along the flow, in the sense of the theorem that follows. The main idea is to appropriately take the average of an arbitrary norm along the flowlines for sufficiently long time. This relies on an averaging technique of Holmes \cite{holmes}, reiterated by Hirsch-Pugh-Shub \cite{invariant}, which is used to prove the existence of adapted norms on the invariant bundles of Anosov flows. The theorem below is a strong refinement of their result to show that moreover, the variations of such norm along the flow can also be controlled to be arbitrary small. Gourmelon \cite{gour} uses the same technique for projectively Anosov flows (more generally, flows with dominated splitting) and our proof is closer in spirit to his computations. 
 
 In the following theorem, $\overset{\text{exp.}}{\longrightarrow}$ denotes exponential convergence as $T \longrightarrow \infty$.

 \begin{theorem}\label{expunif}
 Suppose $X^t$ is a complete flow on a manifold $M$ of arbitrary dimension, inducing an action $X^t_*$ on the line bundle $E\rightarrow M$ equipped with a fiberwise norm $||.||$ as above.
There exists a deformation $||.||_T$ for $T\in [0,\infty)$ of $||.||$, {\em exponentially uniformizing $||.||$} along $X^t$, by which we mean that $||.||_0=||.||$ and the induced expansion rates of $||.||_T$, defined by
$$r_T:=\partial_t \cdot \ln{\big(||X_*^t ||_T\big)}\bigg|_{t=0}:M\rightarrow \mathbb{R},$$
satisfy
 $$r_{T}-\frac{1}{T}[\int_0^T r_0(\tau)d\tau ]\mbox{\Large$\overset{\text{exp.}}{\longrightarrow}$} 0$$
 and
$$X^{(j)}\cdot r_{T}-\frac{1}{T}[X^{(j-1)}\cdot r_{0}(T)-X^{(j-1)}\cdot r_{0}]\mbox{\Large$\overset{\text{exp.}}{\longrightarrow}$} 0$$
for $1\leq j\leq N-1$, where $X^{(j)}$ refers to the $j$th derivative along $X^t$ and $r_0(t)=r_0\circ X^t$.
\end{theorem}
 
 \begin{proof}
 Let $r=r_0$ be the expansion rate induced from the norm $||.||$ on $E$. That is
 $$r:=\partial_t \cdot \ln{\big(||X_*^t ||\big)}\bigg|_{t=0}:M\rightarrow \mathbb{R}.$$
 
 We start by defining the following functions for arbitrary $T,\epsilon_T>0$ ($T$ should be thought as very large and $\epsilon_T$ as very small):
 $$\begin{cases}
 q:=exp(-\epsilon_T+\frac{1}{T}\int_0^Tr(\tau)d\tau) \\
 
 s:=exp(+\epsilon_T+\frac{1}{T}\int_0^Tr(\tau)d\tau)
 \end{cases},$$
 where $r(t)=r \circ X^t:M\rightarrow \mathbb{R}$ (note $q,s:M\rightarrow \mathbb{R}$).
 
  We then define the functions $Q_t,S_t:M\rightarrow \mathbb{R}$ for arbitrary $t$.
  $$\begin{cases}
 Q_t:=exp(\int_0^t  \ln{q(u)}  du)=exp(\int_0^t \{-\epsilon_T +\frac{1}{T}\int_0^Tr(u+\tau)d\tau \}du) \\
 
 S_t:=exp(\int_0^t  \ln{s(u)} du)=exp(\int_0^t \{+\epsilon_T +\frac{1}{T}\int_0^Tr(u+\tau)d\tau \}du)
 \end{cases}$$
 and can easily check that $Q_0=S_0\equiv 1$ and $Q_t,S_t$ are cocycles, i.e. $Q_{t+k}=Q_tQ_k(t)$ and $S_{t+k}=S_tS_k(t)$, where $Q_k(t)=Q_k\circ X^t$ and similarly, for $S_k(t)$.
 
 \begin{lemma}
 We have the {\em domination relation} $q\prec X_*|_E\prec s$. By this, we mean that for any $v\in E$, the quantities $\frac{||X_*^t(v) ||}{S_t}$ and $\frac{||X_*^t(v) ||}{Q_t}$ converge exponentially to zero everywhere, when $t\rightarrow +\infty$ and $t\rightarrow -\infty$, respectively.
 \end{lemma}
 
 \begin{proof}
 For $X_*|_E\prec s$, we need to show $\frac{||X_*^t||}{S_t}$ goes exponentially to $0$ as $t\rightarrow \infty$, i.e. $\lim_{t\rightarrow \infty} \ln{\frac{||X_*^t||}{S_t}}=0$. Compute 
 $$\ln{\frac{||X_*^t||}{S_t}}=\int_0^t r (\tau)d\tau-\epsilon_Tt-\int_0^t \{ \frac{1}{T} \int_0^T r(u+\tau)du\}d\tau$$
 $$=\int_0^t r (\tau)d\tau-\epsilon_Tt-\int_0^T \frac{\tau}{T}r(\tau)d\tau -\int_T^t r(\tau)d\tau -\int_0^T \frac{T-\tau}{T} r(t+\tau)d\tau$$
 $$=-\epsilon_Tt+\int_0^T \frac{T-\tau}{T}(r(\tau)-r(t+\tau))d\tau,$$
 where the second equality is achieved after some integration by parts and substitutions.
 
 Now, note that 
 $$|\int_0^T \frac{T-\tau}{T}(r(\tau)-r(t+\tau))d\tau|\leq (\max{r}-\min{r})\int_0^T \frac{T-\tau}{T}d\tau=\frac{T}{2}(\max{r}-\min{r}).$$
 Such term being bounded implies $\lim_{t\rightarrow \infty} \ln{\frac{||X_*^t||}{S_t}}=0$ as needed. We can similarly show $q\prec X_*|_E$.
 \end{proof}
 
 Now we can define the following norm on $E$ and we will show that such norm has the desired properties for sufficiently large $T$ and small $\epsilon_T$. Let $v \in E$:
 $$||v||_T^2:= \int_{-\infty}^0 \frac{||X_*^t(v)||^2}{Q_t^2}dt+\int_0^{+\infty} \frac{||X_*^t(v)||^2}{S_t^2}dt.$$

The goal is to show that as $T\rightarrow \infty$ the expansion rate of the above norm converges to the {\em forward Lyapunov exponent} at each point and its derivatives along the flow approach $0$. First consider the following corollary from the proof of the above lemma.

\begin{corollary}
We have $ \epsilon_T  e^{-T(\max{r}-\min{r})} \leq B_T:=\frac{||v||^2}{||v||_T^2}\leq \epsilon_T e^{T(\max{r}-\min{r})}$ for any $v \in E$.
\end{corollary}

\begin{proof}
The computation of the lemma shows that $$e^{-2\epsilon_T t-T(\max{r}-\min{r})}\leq \frac{||X_*^t||^2}{S^2_t}\leq e^{-2\epsilon_T t+T(\max{r}-\min{r})}$$
$$\Rightarrow e^{-T(\max{r}-\min{r})} \int_0^\infty e^{-2\epsilon_T t}dt \leq \int_0^\infty \frac{||X_*^t||^2}{S^2_t} dt \leq e^{T(\max{r}-\min{r})} \int_0^\infty e^{-2\epsilon_T t}dt$$
$$\Rightarrow \frac{e^{-T(\max{r}-\min{r})}}{2\epsilon_T} \leq \int_0^\infty \frac{||X_*^t||^2}{S^2_t} dt \leq  \frac{e^{T(\max{r}-\min{r})}}{2\epsilon_T}.$$
doing the same computation for $\frac{||X_*^t||^2}{Q^2_t}$ and adding to the above gives the claim.
\end{proof}
 
Now note that $r_{T}(k):=\frac{1}{2}\partial_k\cdot\ln{|| X_*^k(v) ||^2_T}$ for $v\in E$ is the expansion rate with respect to the norm $||.||_T$ and compute
 $$|| X_*^k(v) ||^2_T=Q_k^2 \int_{-\infty}^k  \frac{||X_*^t(v)||^2}{Q_t^2}dt+S_k^2 \int_{k}^\infty  \frac{||X_*^t(v)||^2}{S_t^2}dt$$
 $$\Rightarrow \partial_k \cdot || X_*^k(v) ||^2_T=\partial_k\cdot (Q_k^2) \int_{-\infty}^k  \frac{||X_*^t(v)||^2}{Q_t^2}dt+|| X_*^k(v) ||^2_T+\partial_k\cdot(S_k^2) \int_{k}^\infty  \frac{||X_*^t(v)||^2}{S_t^2}dt-|| X_*^k(v) ||^2_T$$
 $$=\partial_k\cdot (Q_k^2) \int_{-\infty}^k  \frac{||X_*^t(v)||^2}{Q_t^2}dt+\partial_k\cdot(S_k^2) \int_{k}^\infty  \frac{||X_*^t(v)||^2}{S_t^2}dt.$$
 
 Also, we have
 $$\partial_k \cdot Q^2_k=2\{ -\epsilon_T+\frac{1}{T}\int_0^Tr(k+\tau)d\tau\}Q^2_k$$
 and
$$\partial_k \cdot S^2_k=2\{ +\epsilon_T+\frac{1}{T}\int_0^Tr(k+\tau)d\tau\}S^2_k,$$
implying $$ \partial_k \cdot || X_*^k(v) ||^2_T \big|_{k=0}=\big[2\epsilon_T A_T+ \frac{2}{T}\int_0^Tr(\tau)d\tau\big] ||v||^2_T,$$
where $A_T:=\frac{\int_{0}^\infty  \frac{||X_*^t(v)||^2}{S_t^2}dt - \int_{-\infty}^0  \frac{||X_*^t(v)||^2}{Q_t^2}dt}{||v||_T^2}$ and in particular, $|A_T|\leq 1$. This means
\begin{equation}r_{T}=\epsilon_T A_T+ \frac{1}{T}\int_0^Tr(\tau)d\tau. \end{equation}
Note
$$X\cdot  r_{T} (k)= \partial_k \cdot r_{T}(k)=\frac{1}{2}\partial_{kk}^2\cdot \ln{||X_*^k(v)||_T^2}=\frac{1}{2} \partial_k\cdot \frac{\partial_k \cdot ||X_*^k(v)||_T^2}{||X_*^k(v)||_T^2}$$ 
$$=\frac{(\partial_{kk}^2 \cdot ||X_*^k(v)||_T^2)||X_*^k(v)||_T^2-(\partial_k \cdot ||X_*^k(v)||_T^2)^2}{2||X_*^k(v)||_T^4}.$$

So, we need to compute 
$$\partial_{kk}^2 \cdot ||X_*^k(v)||_T^2=\partial^2_{kk}\cdot (Q_k^2) \int_{-\infty}^k  \frac{||X_*^t(v)||^2}{Q_t^2}dt+\partial^2_{kk}\cdot(S_k^2) \int_{k}^\infty  \frac{||X_*^t(v)||^2}{S_t^2}dt-4\epsilon_T||X_*^k(v)||^2,$$
where we have used the fact that $\partial_k\cdot \ln{(\frac{Q^2_k}{S^2_k})}=\partial_k\cdot (-4\epsilon_Tk)=-4\epsilon_T$. Also, we have 

$$\partial_{kk} \cdot Q^2_k=\big[4( -\epsilon_T+\frac{1}{T}\int_0^Tr(k+\tau)d\tau)^2 +2\frac{r(k+T)-r(k)}{T} \big]Q^2_k$$
$$\partial_{kk} \cdot S^2_k=\big[4( +\epsilon_T+\frac{1}{T}\int_0^Tr(k+\tau)d\tau)^2 +2\frac{r(k+T)-r(k)}{T} \big]S^2_k,$$
which implies 
 $$\partial_{kk}^2 \cdot ||X_*^k(v)||_T^2\big|_{k=0}=\big[ 4\epsilon_T^2+4(\frac{1}{T}\int_0^Tr(\tau)d\tau)^2 +2\frac{r(T)-r}{T}\big]||v||_T^2 + \big[ \frac{8\epsilon_T}{T} \int_0^Tr(\tau)d\tau\big]A_T||v||^2_T-4\epsilon_T||v||^2.$$
 
 Therefore, we have
 $$2\partial_k \cdot r_{T}(k) \big|_{k=0} =\big[ 4\epsilon_T^2+4(\frac{1}{T}\int_0^Tr(\tau)d\tau)^2 +2\frac{r(T)-r}{T}\big] $$
 $$+ \big[ \frac{8\epsilon_T}{T} \int_0^Tr(\tau)d\tau\big]A_T-4\epsilon_TB_T -\big[ 2\epsilon_T A_T+\frac{2}{T}\int_0^T r(\tau)d\tau \big]^2,$$
 implying
 \begin{equation} X\cdot r_{T}=2\epsilon_T^2+\frac{r(T)-r}{T}-2\epsilon_T B_T -2\epsilon_T^2 A^2_T\  .\end{equation}
 
 Now, considering the Corollary above, if we let $\epsilon_T:=e^{-T(\max{r}-\min{r})}$, we'll have $B_T\leq 1$ and $\epsilon_T\rightarrow 0$ as $T\rightarrow 0$ (note that in the case $(\max{r}-\min{r})=0$, there is nothing to prove. So can assume $(\max{r}-\min{r})>0$). Equations (1) and (2) gives that 
 $$\lim_{T\rightarrow \infty} r_{T} =\lim_{T\rightarrow \infty}\frac{1}{T}\int_0^T r(\tau)d\tau $$
 and
 $$\lim_{T\rightarrow \infty} X\cdot r_{T}=0. $$
 
 Note that the first vanishing limit implies that whenever the Lyapunov exponent exists, it is equal to $\{\lim_{T\rightarrow \infty} r_{T}\}$.
 
Thanks to the exponential nature of the convergence, this argument can be inductively be applied to show that the higher derivatives $X^k\cdot r_u$ for $ k\in \mathbb{N}$ also uniformly vanish. More specifically, we have shown
$$r_{T}-\frac{1}{T}[\int_0^T r(\tau)d\tau ]\mbox{\Large$\overset{\text{exp.}}{\longrightarrow}$} 0$$
and 
$$X\cdot r_{T}-\frac{1}{T}[r(T)-r_u(0)] \mbox{\Large$\overset{\text{exp.}}{\longrightarrow}$} 0.$$
Thanks to the exponential convergence in the above, this can in fact be further improved to
$$X^{(j)}\cdot r_{T}-\frac{1}{T}[X^{(j-1)}\cdot r(T)-X^{(j-1)}\cdot r(0)] \mbox{\Large$\overset{\text{exp.}}{\longrightarrow}$} 0$$
for $1\leq j<N$, which then yields the uniform convergence
$$X^{(j)}\cdot r_{T}  \mbox{\Large$\overset{\text{unif.}}{\longrightarrow}$} 0.$$

  \end{proof}
  
 \begin{remark}\label{perexp}
 In particular, Theorem~\ref{expunif} can be applied to the strong invariant bundles $E^s$ or $E^u$ of an Anosov 3-flow $X^t$ and it is useful to remark a few implications in this case.

We note that if $\gamma$ is a periodic orbit for $X^t$ with period $T$, the (forward and backward) Lyapunov exponents of points on $\gamma$ are always defined and determined by the eigenvalues of the associated return map. More specifically, such eigenvalues $\lambda_u$ and $\lambda_s$ corresponding to the eigenspaces $E^u$ and $E^s$, respectively, satisfy $|\lambda_u|=e^{\int_0^T r_u(\tau)\ d\tau}$ and $|\lambda_s|=e^{\int_0^T r_s(\tau)\ d\tau}$ (see Proposition~\ref{exppro}). Therefore, the associated Lyapunov exponents are $\bar{r}_u=\frac{\ln{|\lambda_u|}}{T}$ and $\bar{r}_s=\frac{\ln{|\lambda_s|}}{T}$, respectively, and in particular, the expansion rates $r_{u,T}$ and $r_{s,T}$ of the deformations given by Theorem~\ref{expunif} limit to these values on $\gamma$ when $T\rightarrow \infty$.
 
 If $X^t$ is furthermore transitive, the Lyapunov exponents are constant on a set of full measure, and therefore, the limits $\lim r_{u,T}$ and $\lim r_{s,T}$ exist, and equal the constants $h$ and $-h$, respectively, almost everywhere. The quantity $h$ is independent of the norm and in fact measures the {\em Liouville entropy} of the flow (see \cite{hoz5}).
 \end{remark}


We note that Theorem~\ref{expunif} completes the proof of our main Theorem~\ref{main}. However, Theorem~\ref{expunif} in fact provides a much stronger convergence result and in particular, motivates the following definition (see Theorem~\ref{mains} and the discussion at the end of Section~\ref{s3}). 

\begin{definitionp}\label{snorm}
Using the above notation, a norm on the expanding or contracting vector bundle $E\rightarrow M$ is called {\em strongly adapted}, if its derivative along the flow is also an adapted norm on $E\rightarrow M$. This is equivalent to its expansion $r$ rate satisfying $r^2+X\cdot |r|> 0$.
\end{definitionp}

We have the following corollary.
\begin{corollary}\label{strnormex}
 Let $X^t$ and $E\rightarrow M$ be as in Theorem~\ref{expunif}. Then, $E$ admits a strongly adapted norm.
\end{corollary}

Similar to adapted contact forms (see Remark~\ref{strongfilter}), this provides a filtration of the space of adapted norms, given a flow action as in Theorem~\ref{expunif}, and our result indicates that all these sub classes of norms on $E$ are non-empty. This is a refinement of the results by Hirch-Pugh-Shub \cite{invariant} and Gourmelon \cite{gour} in this setting, by extending control over the variations of an expanding or contracting norm along the flow.


\section{The space of adaptations}\label{s5}

The goal of this section is to investigate the space of the geometric structures resulting from our characterization of Anosov flows, i.e. {\em strongly adaptations}, as we will soon define precisely. One of our chief purposes is to show that the space of such geometric models is homotopy equivalent to the space of Anosov flows and hence, the deformation invariants of these geometric structures can be used to study the space of Anosov flows. For our purposes in this paper, we consider Anosov 3-flows up to unoriented reparametrization, as the underlying contact geometry is unaffected under such operation. In terms of the generating vector field, this can be viewed as (recall that vector fields and other geometric objects are assumed $C^\infty$ in this paper, unless stated otherwise)
$${\mathcal{A}}(M):=\frac{\{ X\in\chi (M) |X \text{ is Anosov } \}\  }{X\sim fX \text{ for } f:M\rightarrow \mathbb{R}/\{0\}},$$
where $\chi (M)$ is the space of ($C^\infty$) vector fields in $M$. We denote the equivalence class of an Anosov vector field by $\langle X \rangle\in\chi (M)$. Also note that $\mathcal{A} (M) \subset \chi(M) / \sim$ is an open subset.

We want to study the space of strong adaptations. By this, we mean the study of the space $$\mathcal{SA}^+(M):=\bigg\{ (\alpha_+,\langle X \rangle)| \alpha_+ \text{ strongly adapted to }\langle X \rangle\in {\mathcal{A}}(M) \bigg\}/ (\mathbb{R}\backslash\{0\},.),$$
where we identify contact forms, up to (possibly negative) constant multiplication.
Of course, one might want to study the strongly adapted contact geometry from the viewpoint of the underlying Reeb vector field or bi-contact structures (unlike the above definition in terms of contact forms), as there are other formulations of strong adaptations (see Proposition~\ref{strad}). We will soon see (Corollary~\ref{saspace}) that the space $\mathcal{SA}^+(M)$ can also be defined in terms of those geometric objects.

We need to define the following spaces for what follows. 
\begin{itemize}
\item $\mathcal{C}^b(M)$:= the space of bi-contact structures supporting some Anosov flow. 
Note that there is a natural projection 
$\begin{cases}
\mathcal{C}^b(M)\rightarrow\mathcal{A}(M) \\
(\xi_-,\xi_+)\mapsto \xi_-\cap\xi_+
\end{cases}$.
\item $\mathcal{C}^\pm(M):=\big\{ (\xi_\pm,\langle X \rangle)| \xi_\pm \text{ a positive/negative contact structure supporting } \langle X \rangle\in\mathcal{A}(M)\big\}$. Recall that a negative or positive contact structure supporting a Anosov flow $X^t$ is one appearing in some supporting bi-contact structure (see Definition~\ref{suppdef}). Therefore, there are natural projections from $\mathcal{C}^b(M)$ onto $\mathcal{C}^+(M)$ and $\mathcal{C}^-(M)$, as well as projections from $\mathcal{C}^+(M)$ and $\mathcal{C}^-(M)$ onto $\mathcal{A}(M)$, extracting the supported vector field (up to unoriented reparametrization).

\end{itemize}

 We start with the following observation, which indicates the subtlety in understanding the space of strongly adapted contact forms.

\begin{remark}\label{phen}
In general, we can't fix an arbitrary supporting bi-contact structure $(\xi_-,\xi_+)$ and expect to find a Reeb vector field $R_{+}$ (of $\xi_+$) such that $R_{+}\subset \xi_-$. To see this, choose a strongly adapted $\alpha_+$ and let $\xi_-:=\langle R_{+},X\rangle$. Then by Theorem~\ref{vol}, there is no Reeb vector field $R_{-}$ of $\xi_-$ with $R_{-}\subset \xi_+$, unless $X$ is incompressible. This in fact provides the bi-contact geometric characterization of incompressibility in Anosov flows \cite{hoz4}.
\end{remark}

In a series of lemmas, we observe various properties of the space of strongly adapted contact geometry $\mathcal{SA^+}(M)$ and in particular, observe some of its symmetries. These are mostly based on Gray's theorem and an application of the {\em Moser's technique}, which allow us to deform supporting contact structures via isotopies along the flow.

It is important to notice that the space of smooth functions $f:M\rightarrow \mathbb{R}$ has a natural smooth action on $\mathcal{SA^+}(M)$.

\begin{lemma}
If $\alpha_+$ is strongly adapted, so is $\alpha_{+,f}:=X^{f*}\alpha_+$, for any $f:M\rightarrow \mathbb{R}$.
\end{lemma}

\begin{proof}
That is since the Reeb vector field for $\alpha_{+,f}$ satisfies $R_{\alpha_{+,f}}=X^{-f}_* R_{\alpha_+}$ and therefore,
$$\xi_{-,f}:=\langle R_{\alpha_{+,f}},X\rangle=\langle X_*^{-f} R_{\alpha_+},X\rangle=X_*^{-f}\xi_-,$$
where $\xi_-=\langle R_{\alpha_+},X\rangle$.
\end{proof}

This allows us to show that any supporting $\xi_+$ or $\xi_-$ for an Anosov flow can be part of a strongly adapted supporting bi-contact structures, which will be used later to show that the strong adaptations can be fibered over the space of positive and negative supporting contact structures.

\begin{lemma}\label{grayf}
For any supporting $\xi_+$, there exists a strongly adapted $\alpha_+$ with $\ker{\alpha_+}=\xi_+$.
\end{lemma}

\begin{proof}
Let $\bar{\alpha}_+$ be an arbitrary strongly adapted contact form. Then, there exists a function $f:M\rightarrow \mathbb{R}$ such that $\xi_+=X_*^{-f}\bar{\xi}_+$, where $\bar{\xi}_+:=\ker{\bar{\alpha}_+}$. This is thanks to the fact that the two supporting positive contact structures can be homotoped through supporting positive contact structures, as discussed in Section~5 of \cite{hoz3} (see Theorem~1.5 in that paper), as well as the standard use of {\em Gray's theorem} in contact geometry, which guarantees such homotopy of contact structures can be achieved by an isotopy generated by an appropriate scaling of the vector field in their intersections, i.e. $X$ (see Theorem~4 in \cite{massot}).
Thus, by the previous lemma, we have $\alpha:=X^{f*}\bar{\alpha}_+$ is strongly adapted and $\ker{\alpha_+}=\xi_+$.
\end{proof}

\begin{lemma}
For any supporting $\xi_-$, there exists a strongly adapted $\alpha_+$ with $R_{\alpha_+}\subset\xi_-$.
\end{lemma}

\begin{proof}
Let $\bar{\alpha}_+$ be an arbitrary strongly adapted contact form and let $\bar{\xi}_-=\langle R_{\bar{\alpha}_+},X\rangle$. As in the proof of Lemma~\ref{grayf}, there exists $f:M\rightarrow \mathbb{R}$ such that $X^{-f}_*\bar{\xi}_-=\xi_-$. Let $\alpha_+:=X^{f*}\bar{\alpha}_+$ and note that
$$\langle R_{\alpha_+},X\rangle=\langle X^{-f}_*R_{\bar{\alpha}_+},X\rangle=X_*^{-f}\langle R_{\bar{\alpha}_+},X\rangle=X_*^{-f}\bar{\xi}_-=\xi_-.$$
\end{proof}

Moreover, we can show that a strongly adapted Reeb dynamics $R_{+}$ is determined, up to constant scaling, by fixing $(\xi_-,\xi_+)$ (in the case of existence, see Remark~\ref{phen}). In the following $\pi_+^{-1}\langle X\rangle$ is the space of positive contact structures supporting a fixed $X$.

\begin{lemma}\label{moserlemma}
Suppose $\alpha_+$ and $\bar{\alpha}_+$ be such that $\ker{\alpha_+}=\ker{\bar{\alpha}_+}\in \pi_+^{-1}\langle X\rangle$ and $\langle R_{\alpha_+},X\rangle=\langle R_{\bar{\alpha}_+},X\rangle$. Then, $\alpha_+=C\bar{\alpha}_+$ for some constant $C\in\mathbb{R}$.
\end{lemma}

\begin{proof}
Let $(\xi_-,\xi_+)=(\langle R_{\alpha_+},X\rangle,\ker{\alpha_+})$.
Write $\bar{\alpha}_+=e^h\alpha_+$ for some $h:M\rightarrow \mathbb{R}$. A simple computation shows that $R_{\bar{\alpha}_+}=e^{-h}R_{\alpha_+}+X_h$ for a unique $X_h \subset \xi_+$ satisfying
\begin{equation}\label{moser}
\iota_{X_h}d\alpha_+ |_{\xi_+}=e^{-h}dh|_{\xi_+} \Longleftrightarrow \iota_{X_h}d\bar{\alpha}_+ |_{\xi_+}=dh|_{\xi_+}.
\end{equation}
So if $R_{\bar{\alpha}_+}\subset \xi_-$, this implies $X_h\subset \xi_-\cap \xi_+$ and hence, $X_h||X$. Therefore, we have 
$$0=d\alpha_+ (X_h,X)=X\cdot h.$$
This is enough to conclude that $h$ is in fact constant. To see this, suppose it is not. Then, there is some $c\in\mathbb{R}$ which is a regular value for $h$ and in particular, $h^{-1}(c)$ is a compact surface. The fact that $X\cdot h=0$, then implies that such compact surface is invariant under $X$. But Anosov flows do not admit any invariant closed surface.
\end{proof}

An important consequence of our computations is the following corollary, giving alternative formulations of the space of strong adaptations.

\begin{corollary}\label{saspace}
We have
$$\mathcal{SA}^+(M)\simeq\bigg\{(R_+,\langle X \rangle)| R_+ \text{ is a Reeb vector field strongly adapted to }\langle X \rangle\in\mathcal{A}(M) \bigg\}/(\mathbb{R}\backslash\{0\},.)$$
$$ \simeq \bigg\{((\xi_-,\xi_+)| (\xi_-,\xi_+)\in \mathcal{C}^b(M) \text{ is a bi-contact structure strongly adapted to }[\xi_-\cap\xi_+] \bigg\}.$$
\end{corollary}

\begin{proof}
To see the first equivalence, note that a strongly adapted $R_+$ determines $\alpha_+$ via $\ker{\alpha_+}=\langle [R_+,X],X\rangle $ and $\alpha_+(R_+)=1$. The second equivalence is a direct consequence of Lemma~\ref{moserlemma} and Equation~\ref{moser} in particular. 
\end{proof}

Our next goal is to establish that when fixing an Anosov flow direction $\langle X \rangle$, the space of strong adaptations is contractible. Ideally, we would hope that such space is convex. But that does not seem to be the case. In the following however, we show that it fibers over the space of Anosov flows, with contractible fibers. This is called an {\em acyclic fibration}. In particular, $\mathcal{SA}^+(M)$ is homotopy equivalent to the space of Anosov flows on $M$.

\begin{theorem}\label{fiber}
The map $$\begin{cases}
\pi_A:\mathcal{SA}^+(M)\rightarrow {\mathcal{A}}(M)\\
(\alpha_{+},\langle X \rangle)\mapsto \langle X \rangle
\end{cases}$$ is an acyclic fibration. 
\end{theorem}

We give two proofs for this theorem, which we believe are both useful. One naturally yields a filtration of the above fibration through the space of supporting contact structures, which might be more natural from the contact geometric viewpoint, and is more compatible with the our high regularity setting. The other proof relies on a filteration through the space of adapted norms on the stable bundle $E^s$ (similarly, filtration through adapted norms on $E^u$ is possible). This approach has corollaries on the theory of adapted norms in Anosov dynamics.

{\em Proof 1.}
It is known that the space of supporting contact structures form an acyclic {\em Serre} fibration over the space of Anosov flows. This was first noted by Massoni \cite{massoni}. We improve this in the following lemma. 
In the following, $C^\infty(M,\mathbb{R})$ is the space of smooth real functions on the underlying manifold $M$.

\begin{lemma}\label{l1}
We have the fiber bundle structure
$$\begin{tikzcd} 
 C^\infty(M;\mathbb{R})  \arrow[r, hook] & \mathcal{C}^+(M) \arrow[d, two heads] \\   &  {\mathcal{A}}(M) 
\end{tikzcd}.$$

The same is true for $\mathcal{C}^-(M)$, as well as $\mathcal{C}^b(M)$ with  $C^\infty(M;\mathbb{R}) \times  C^\infty(M;\mathbb{R}) $ fibers.
\end{lemma}

\begin{proof}
We have the projection $\begin{cases} \pi_+: \mathcal{C}^+(M)\rightarrow \mathcal{A}(M)\\ \pi_+(\xi_+,\langle X \rangle)=\langle X \rangle \end{cases}$. Fix an Anosov vector field $X$ and a supporting positive contact structure $\xi_+$. We can naturally identify $\pi_+^{-1}\langle X\rangle$ with the space of ($C^\infty$) real functions on $M$, which we denote by $C^\infty(M;\mathbb{R})$, via
$$\begin{cases}
C^\infty(M;\mathbb{R})\rightarrow \pi^{-1}_+\langle X\rangle \\
f\mapsto X^{f}_* \xi_+
\end{cases}.$$
Now take a $C^\infty$ vector field $v_+\subset \xi_+$ which is linearly independent with $X$. If $\bar{X}$ be an Anosov vector field $\epsilon$-close to $X$, then $\bar{\xi}_+:=\langle \bar{X},v_+ \rangle$ is a supporting positive contact structure for $\bar{X}$, if $\epsilon$ is taken sufficiently small (thanks to the openness of the contact condition). This results in a canonical identification of $\pi^{-1}_+\langle \bar{X}\rangle$ with $C^\infty(M;\mathbb{R})$ again. This means that the fibration $\pi_+:\mathcal{C}^+(M)\rightarrow \mathcal{A}(M)$ in fact has a locally product structure with fibers equivalent to $\mathcal{C}^+(M)$.
\end{proof}

Given the above, it is enough to show that fixing an Anosov flow direction $\langle X\rangle$, the space of strongly adapted contact forms with a fixed kernel $\xi_+$ is contractible.

\begin{lemma}\label{l2}
Fixing $(\xi_+,\langle X \rangle)\in \mathcal{C}^+(M)$, the set $\{\alpha_+ |\text{ strongly adapted to }\langle X \rangle
, \ker{\alpha_+}=\xi_+  \}/ \{ \pm 1\}$ is logarithmically convex.
\end{lemma}

\begin{proof}
Let $\alpha_+$ and $\bar{\alpha}_+=e^h\alpha_+$ be strongly adapted contact forms for $X$, with Reeb flows $R_+$ and $\bar{R}_+=e^{-h}R_++X_h$, respectively, where $X_h$ is determined via Equation~\ref{moser}. We want to show that $\alpha_{+,\tau}=e^{\tau h}\alpha_+$ is strongly adapted for $\tau\in [0,1]$ (i.e. being {\em logarithmically convex}). Let $R_{+,\tau}=e^{-\tau h}R_+ +X_{\tau h}$ be the Reeb vector field for $\alpha_{+,\tau}$ and notice that Equation~\ref{moser} also implies $e^{\tau h}X_{\tau h}=\tau e^hX_h$. We have
$$\alpha_{+,\tau}\text{ strongly adapted to }X \Longleftrightarrow 0<d\alpha_{+,\tau}(X,[X,R_{+,\tau}])$$
$$\Longleftrightarrow 0<d\alpha_{+,\tau}(X,[X,e^{\tau h}R_{+,\tau}])=d\alpha_{+,\tau}(X,[X,R_++e^{\tau h}X_{\tau h}])$$
$$\Longleftrightarrow 0<d\alpha_+(X,[X,R_+]+[X,e^{\tau h}X_{\tau h}])=d\alpha_+(X,[X,R_{\alpha_+}])+\tau d\alpha_+(X,[X,e^h X_h]).$$

Therefore, $\alpha_{+,\tau}$ is strongly adapted for all $\tau\in[0,1]$, if and only if, $\alpha_{+,0}=\alpha_+$ and $\alpha_{+,1}=\bar{\alpha}_+$ are strongly adapted, completing the proof.
\end{proof}

Now, we have a composition of acyclic fibrations 
$$
\begin{cases}
\mathcal{SA}^+(M) \rightarrow \mathcal{C}^+(M)\rightarrow \mathcal{A}(M) \\
(\alpha_+,\langle X \rangle)\mapsto (\ker{\alpha_+},\langle X \rangle) \mapsto \langle X \rangle
\end{cases}.
$$
This gives the desired acyclic fibration.
\qed

The above proof gives in fact a natural filtration of the fibration map $\pi_A:\mathcal{SA}^+(M)\rightarrow \mathcal{A}(M)$ through $\mathcal{C}^+(M)$. Similarly, a filteration of $\pi_A$ can be given through $\mathcal{C}^-(M)$ via
$$\begin{cases}
\mathcal{SA}^+(M)\rightarrow \mathcal{C}^-(M)\rightarrow \mathcal{A}(M) \\
(\alpha_+,\langle X\rangle)\mapsto (\langle R_{\alpha_+},X \rangle,\langle X \rangle)\mapsto \langle X \rangle
\end{cases},$$
which might be more natural from the Reeb dynamical definition of $\mathcal{SA}^+(M)$ (see Corollary~\ref{saspace}). Finally, for the bi-contact geometric approach, it should be noted that we have an embedding of fibrations via
$$\begin{cases}
\mathcal{SA}^+(M)\rightarrow \mathcal{C}^b(M) \\
(\alpha_+,\langle X \rangle)\mapsto (\langle R_{\alpha_+},X \rangle,\ker{\alpha_+})
\end{cases}.$$

We summarize these observation in the following refinement of the fibration structure claimed in Theorem~\ref{fiber}.

\begin{corollary}\label{refine}
Using the above notation, we have the commutative diagram
$$\begin{tikzcd} 
\mathcal{SA}^+(M)  \arrow[drr, two heads, bend left] \arrow[ddr, two heads, bend right] \arrow[dr, hook] & & \\ & \mathcal{C}^b(M)  \arrow[r, two heads] \arrow[d, two heads] \arrow[dr, two heads]& \mathcal{C}^+(M) \arrow[d, two heads] \\ & \mathcal{C}^-(M) \arrow[r, two heads] &  {\mathcal{A}}(M) 
\end{tikzcd},$$
where any composition of maps from $\mathcal{SA}^+(M)$ to ${\mathcal{A}}(M)$ is equivalent to $\pi_A$.
\end{corollary}

We now present a second proof for Theorem~\ref{fiber} with emphasis on the space of norms induced on the invariant bundles.

{\em Proof 2.}
In this proof, we want to focus on the adapted norms involved. We prove the following intermediate step (see Definition~\ref{snorm}).

\begin{lemma}\label{adnormconvex}
Fixing Anosov $X$, the space of (strongly) adapted norms on $E^u$ is convex.
\end{lemma}

\begin{proof}
Suppose $\alpha_u$ and $\bar{\alpha}_u=h\alpha_u$, where $h:M\rightarrow\mathbb{R}_{>0}$ , induce strongly adapted norms on $E^u$, with expansion rates $r_u$ and $\bar{r}_u$, respectively, and define
$$\alpha_{u,\tau}:=(1-\tau)\alpha_u+\tau\bar{\alpha}_u=[(1-\tau)+\tau h]\alpha_u.$$

To derive the expansion rate of $\alpha_{u,\tau}$, compute
$$\mathcal{L}_X\alpha_{u,\tau}=(1-\tau)r_u\alpha_u+\tau\bar{r}_u\bar{\alpha}_u=[(1-\tau)r_u+\tau h\bar{r}_u]\alpha_u=[\frac{(1-\tau)r_u+\tau h\bar{r}_u}{(1-\tau)+\tau h}]\alpha_{u,\tau},$$
implying that the expansion rate of $\alpha_{u,\tau}$ can be written as
\begin{equation}\label{exps}
r_{u,\tau}=\frac{(1-\tau)r_u+\tau h\bar{r}_u}{(1-\tau)+\tau h}.
\end{equation}
Note that the above formula confirms the fact that the space of adapted norms is convex, i.e. $r_u,\bar{r}_u>0$ implies $r_{u,\tau}>0$. To prove the same for the space of strongly adapted norms, compute
$$X\cdot r_{u,\tau}=\frac{[(1-\tau)X\cdot r_u+\tau (X\cdot \bar{r}_u)h+\tau\bar{r}_uX\cdot h]}{[(1-\tau)+\tau h]}-\frac{[(1-\tau)r_u+\tau h\bar{r}_u](\tau X\cdot h)}{[(1-\tau)+\tau h]^2}$$
and therefore, we write (using Equation~\ref{exps})
$$[(1-\tau)+\tau h]^2(r^2_{u,\tau}+X\cdot r_{u,\tau})$$
$$=[(1-\tau)r_u+\tau h\bar{r}_u]^2+[(1-\tau)X\cdot r_u+\tau(X\cdot \bar{r}_u)h+\tau\bar{r}_uX\cdot h][(1-\tau)+\tau h]-[(1-\tau)r_u+\tau h\bar{r}_u](\tau X\cdot h)$$
$$=(1-\tau)^2[r_u^2+X\cdot r_u]+\tau^2[h^2\bar{r}_u^2+h^2X\cdot \bar{r}_u+\bar{r}_uhX\cdot h-\bar{r}_uhX\cdot h ]$$
$$+\tau(1-\tau)[2r_u\bar{r}_u+X\cdot r_u+X\cdot \bar{r}_u+\bar{r}_uX\cdot \ln{h}-r_uX\cdot \ln{h}]$$
$$=(1-\tau)^2[r_u^2+X\cdot r_u]+\tau^2h^2[\bar{r}_u^2+X\cdot \bar{r}_u]+\tau(1-\tau)h[r_u^2+X\cdot r_u+\bar{r}_u^2+X\cdot \bar{r}_u],$$
where in the last line, we have used the fact that $X\cdot \ln{h}=\bar{r}_u-r_u$. This can further simplified to get
\begin{equation}\label{derexps}
r^2_{u,\tau}+X\cdot r_{u,\tau}=\frac{[(1-\tau)^2+\tau(1-\tau)h][r_u^2+X\cdot r_u]+[\tau^2h^2+\tau(1-\tau)h][\bar{r}_u^2+X\cdot \bar{r}_u]}{[(1-\tau)+\tau h]^2}.
\end{equation}

Equation~\ref{derexps} means that $r_{u,\tau}^2+X\cdot r_{u,\tau}$ is a weighted average of $r_{u}^2+X\cdot r_{u}$ and $\bar{r}^2_u+X\cdot \bar{r}_u$, and in particular, $r_u^2+X\cdot r_u ,\bar{r}_u^2+X\cdot \bar{r}_u>0$, if and only if, $r_{u,\tau}^2+X\cdot r_{u,\tau}>0$ for any $s\in [0,1]$, completing the proof.
\end{proof}

As a corollary of the above computations, we can prove the following.

\begin{corollary}
Fixing $\langle X \rangle$ and an adapted $\alpha_s$, the set $W[\alpha_s]:=\{\alpha_u | \alpha_u-\alpha_s \text{ is strongly adapted} \}$ is convex.
\end{corollary}

\begin{proof}
Let $\alpha_+=\alpha_u-\alpha_s\in W[\alpha_s]$ be strongly adapted. Since $\alpha_s$ is adapted, we can assume $r_s\equiv -1$ after a reparametrization of $X$. This means
$$\alpha_+ \text{ strongly adapted } \Leftrightarrow r_u+X\cdot\ln{r_u}-r_s-X\cdot\ln{(-r_s)}=r_u+X\cdot\ln{r_u}+1>0$$
$$\Leftrightarrow r_u^2+X\cdot r_u>-r_u.$$
Now let $\alpha_{u,\tau}:=(1-\tau)\alpha_u+\tau\bar{\alpha}_u$, where $\alpha_u,\bar{\alpha}_u=h\alpha_u \in W [\alpha_s]$. Using Equation~\ref{derexps} from the previous lemma, and noticing that $r^2_u+X\cdot r_u>-r_u$ and $\bar{r}^2_u+X\cdot \bar{r}_u>-\bar{r}_u$ yields
$$[(1-\tau)+\tau h]^2(r^2_{u,\tau}+X\cdot r_{u,\tau})$$
$$=[(1-\tau)^2+\tau(1-\tau)h][r_u^2+X\cdot r_u]+[\tau^2h^2+\tau(1-\tau)h][\bar{r}_u^2+X\cdot \bar{r}_u]$$
$$> [(1-\tau)^2+\tau(1-\tau)h](-r_u)+[\tau^2h^2+\tau(1-\tau)h](-\bar{r}_u)=-[(1-\tau)+\tau h][(1-\tau)r_u+\tau h \bar{r}_u]$$
$$=-[(1-\tau)+\tau h]^2r_{u,\tau},$$
which implies $$r^2_{u,\tau}+X\cdot r_{u,\tau}>-r_{u,\tau},$$ as desired.
\end{proof}

The above argument shows that fixing $\langle X \rangle$, the space of strongly adapted contact forms is fibered over the space of adapted norms on $E^s$ (the choices of $\alpha_s$ with $r_s<0$), with convex fibers. Moreover, the space of (strongly) adapted norms on $E^s$ is convex, which implies the claim about the contractibility of the space.

To show the fibration claim, it is enough to show that the space of adapted norms is a fibration over the space of Anosov flows. Consider an Anosov vector field $X_0$ equipped with an adapted $\alpha_s$ (i.e. $r_s<0$), which after a perturbation can be assumed to be induced from a smooth Riemannian metric on the ambient manifold. Let $X_\tau$ be a family of Anosov vector fields for $\tau \in (-\epsilon,\epsilon)$, where the induced norms $\alpha_{s,\tau}$ on the stable bundles $E^{s,\tau}$ of all $X_\tau$ are adapted. After a scaling, we get a new family of Anosov vector fields $\bar{X}_{\tau}$, whose expansion rate of the stable bundle satisfies $r_{s,\tau}\equiv -1$ for all $\tau\in(-\epsilon,\epsilon)$. The space of adapted norms on the stable bundle $E^{s,\tau}$ then can be identified with $\{ h:M\rightarrow \mathbb{R} | -1+\bar{X}_\tau\cdot h<0\}$.
\qed

The second proof has implications in the study of adapted norms in hyperbolic dynamics.

\begin{corollary}
Fix an Anosov flow $X^t$. The space of strongly adapted norms on $E^u$ (or $E^s$) is convex.
\end{corollary}

This also implies that one can use the refinement of strong adaptations, as discussed in the final comments of Section~\ref{s3}, to achieve a fiberwise convex fibration of contact geometric models for Anosov 3-flows. More specifically, we can consider a sub-class of strong adaptations as
$$\mathcal{SSA}^+(M):=\bigg\{ (\alpha_+,\langle X \rangle)| \alpha_+ :  \text{strongly adapted, }\mathcal{L}_X\alpha_+ : \text{adapted to }\langle X \rangle \bigg\}/ (\mathbb{R}\backslash\{0\},.)$$
and the above computations show the following. Similarly, further refinements of this space can be studied (in particular, see Theorem~\ref{mains} and Remark~\ref{strongfilter}).

\begin{corollary}\label{sfiber}
The map
$$\begin{cases}
\mathcal{SSA}^+(M)\rightarrow \mathcal{A}(M)\\
(\alpha_{+},\langle X \rangle)\mapsto\langle X \rangle
\end{cases}$$
 is a fiberwise convex fibration.
\end{corollary}


\Addresses
\end{document}